\documentclass{article}

\usepackage[preprint]{neurips_2025}

\usepackage[normalem]{ulem}
\usepackage{amsthm}
\usepackage{comment}
\usepackage{amsmath}
 \usepackage{graphicx} 
\usepackage[utf8]{inputenc} 
\usepackage[T1]{fontenc}    
\usepackage{hyperref}       
\usepackage{url}            
\usepackage{booktabs}       
\usepackage{amsfonts}       
\usepackage{nicefrac}       
\usepackage{microtype}      
\usepackage{xcolor}         

\usepackage[vlined,linesnumbered]{algorithm2e}
\usepackage{algpseudocode}
\usepackage{cleveref}
\let\cref\Cref

\newtheorem{theorem}{Theorem}[section]

\newtheorem{prop}[theorem]{Proposition}

\newtheorem{rmk}[theorem]{Remark}

\newcommand{\X}{X}

\newcommand{\field}{\mathbb{F}}
\newcommand{\caB}{\mathcal{B}}

\newcommand{\Hfunc}{\mathrm{H}}
\newcommand{\rmH}{\mathrm{H}}
\newcommand{\Xfunc}{\mathbb{X}}
\newcommand{\Yfunc}{\mathbb{Y}}

\newcommand{\VR}{\mathbf{VR}}

\newcommand{\cupprod}{\mathbf{cup}}

\newcommand{\bsigma}{\boldsymbol{\sigma}}

\newcommand{\udim}{k}
\newcommand{\card}{\operatorname{card}}

\newcommand{\Top}{\mathbf{Top}}

\newcommand{\Int}{\mathbf{Int}}

\title{Doughnut or Mickey Mouse? Detecting Toroidal Structure in Data through Persistent Cup-Length}

\author{
  Ekaterina S. Ivshina\thanks{\texttt{eivshina@g.harvard.edu}} \\
  John A. Paulson School of Engineering and Applied Sciences\\
  Harvard University\\
   \\
  \And
  Galit Anikeeva \\
  Center for Theoretical Physics - a Leinweber Institute \\ MIT \\
  \AND
  Ling Zhou\thanks{\texttt{ling.zhou@duke.edu}} \\
  Department of Mathematics \\
  Duke University  
}

\begin{document}

\maketitle

\begin{abstract}
 Understanding the structure of high-dimensional data is fundamental to neuroscience and other data-intensive scientific fields. While persistent homology effectively identifies basic topological features such as "holes," it lacks the ability to reliably detect more complex topologies, particularly toroidal structures, despite previous heuristic attempts. To address this limitation, recent work introduced persistent cup-length, a novel topological invariant derived from persistent cohomology. In this paper, we present the first implementation of this method and demonstrate its practical effectiveness in detecting toroidal structures, uncovering topological features beyond the reach of persistent homology alone. Our implementation overcomes computational bottlenecks through strategic optimization, efficient integration with the Ripser library, and the application of landmark subsampling techniques. By applying our method to grid cell population activity data, we demonstrate that persistent cup-length effectively identifies toroidal structures in neural manifolds. Our approach offers a powerful new tool for analyzing high-dimensional data, advancing existing topological methods.
\end{abstract}

\section{Introduction}\label{sec:intro}
Neural representations are increasingly understood to inhabit low-dimensional, nonlinear “manifolds” within the brain \cite{Gallego2018, rybakken2018decodingneuraldatausing,spreemann2015usingpersistenthomologyreveal, kang2021evaluating,Gardner2022}. 
Topological Data Analysis (TDA) provides a powerful framework for understanding the “shape” of such representations by capturing topological features such as connected components, loops, and voids via persistent homology \cite{frosini1990distance,frosini1992measuring,robins1999towards,zomorodian2004computing,cohen2007stability,edelsbrunner2008persistent,carlsson2009topology,carlsson2020persistent}. Previous works used persistent homology heuristically to detect structure in neural data. For example, to detect toroidal structure in neural activity,  researchers looked for two persistent 1-dimensional cycles (loops) and one persistent 2-dimensional cycle (void) \cite{spreemann2015usingpersistenthomologyreveal,Gardner2022}. However, persistent homology focuses on the additive structure of homology groups and can fail to distinguish spaces with the same Betti numbers but fundamentally different topological structures. For example, persistent homology cannot distinguish between a torus (colloquially called a doughnut) and a wedge sum of two circles and a sphere (colloquially called a "Mickey Mouse"); both spaces have two persistent 1-dimensional loops and one persistent 2-dimensional void (see \cref{fig:torus-vs-bunny-persistent}). 

To overcome these limitations, recent work has explored the use of the additional multiplicative structure in cohomology, carried via the cup product \cite{huang2005cup, yarmola2010persistence, 
gakhar2020, belchi2021a, 
medina2022per_st,contreras2022persistent,contessoto2022persistentcuplength, 
dey2024cupproductpersistenceefficient,memoli2024persistent}. However, the adoption of such invariants in TDA has been limited because of computational challenges. 
We contribute to this line of work by not only providing the first practical implementation of a cup product-based persistence algorithm introduced in \cite{contessoto2022persistentcuplength}, but also reinterpreting this computation and linking it to detecting toroidal structure in data.

The cup product encodes interactions among cohomology classes, giving rise to the \emph{cup-length} invariant, which captures higher-order structure. 
Intuitively, the cup-length of a space measures how many cohomology classes of positive degree can be multiplied together (via the cup product) before the result vanishes. It reflects the richness of the cohomology ring beyond just the ranks given by Betti numbers. For example, a cup-length of 2 means there exist two cohomology classes whose product is nontrivial, yet the cup product of any three such classes is zero.    A precise definition of cup-length is provided in \cref{subsub:cohomology and cup length}. 

The cup-length invariant can help distinguish topologically distinct spaces that share the same homology. For instance, the torus \(T^2\) and the wedge sum space \(S^1 \vee S^2 \vee S^1\) share identical homology groups yet differ in cup-lengths. The torus \(T^2\) has two linearly independent 1-dimensional cohomology classes whose cup product yields a nontrivial generator in dimension 2, resulting in a cup-length of 2; each individual class, nonetheless, has trivial self-cup product. In contrast, \(S^1 \vee S^2 \vee S^1\) also has two linearly independent 1-dimensional cohomology classes, but the cup product of any pair vanishes, yielding a cup-length of 1. 

Interestingly, a cup-length of 2, combined with vanishing self-cup products, can be used to detect  toroidal structure. Specifically, we say that a space \( X \) has a toroidal component if its cohomology ring contains a subring isomorphic to the cohomology ring of \( T^2 \). In the special case where \( X \) is a connected, closed, orientable surface, the classification theorem of surfaces implies that if \( X \) has a toroidal component, then it is a connected sum of tori. 
This provides a topological interpretation of the notion of a toroidal component. For details, see Section \ref{sec:toroidal-theory}.

This method of detecting toroidal structure can be applied in practice when analyzing data at multiple scales of geometric detail, using intermediate results from the algorithm computing the \emph{persistent cup-length}  invariant proposed in \cite{contessoto2022persistentcuplength}. While this algorithm computes $\ell$-fold cup products of representative cocycles, its final output reports only the persistence intervals over which cup-length is nonzero, omitting the specific cocycles whose cup product results in nontrivial cup-length. However, this information—which cocycles participate in the nontrivial cup products—carries significant geometric meaning, providing finer-grained insights that are not accessible from the cup-length invariant alone. For example, in detecting toroidal components, it is crucial to identify whether a cup-length of 2 arises from two distinct 1-dimensional cocycles, as this reflects the multiplicative structure characteristic of a torus.

In this work, we present the first implementation and applications of the persistent cup-length algorithm with polynomial-time guarantees \cite{contessoto2022persistentcuplength}, reinterpreting the cup-length computation and establishing theoretical results that link  cup-length to toroidal structure. 
Our implementation overcomes computational bottlenecks through strategic optimization, efficient integration with the Ripser library \cite{bauerGithub, ctralie2018ripser, Bauer2021Ripser}, and the application of landmark subsampling techniques, utilizing the DREiMac library \cite{Perea2023}. Our contributions are:

\begin{enumerate}
    \item An implementation and interpretation of the persistent cup-length algorithm for both simplexwise filtrations (adding simplices sequentially) and metric-induced filtrations  (e.g., Vietoris–Rips). Our code is available on GitHub: \href{https://github.com/eivshina/persistent-cup-length}{https://github.com/eivshina/persistent-cup-length}. 
    \item A theoretical guarantee that a cup-length of 2 in a connected, closed, orientable surface implies the existence of a toroidal component.
    \item Experimental validation of the algorithm using synthetic datasets and grid cell simulations.
    \item An analysis of grid cell recordings, providing rigorous detection of toroidal topology and offering stronger evidence for continuous attractor network (CAN) models underlying spatial navigation.  
\end{enumerate} 

Our work offers a novel topological tool for understanding and leveraging the structure of data. The remaining limitations include computational complexity for large datasets, sensitivity to filtration parameters, and, in general settings, dependence on the choice of cocycle representatives \cite[Example 18]{contessoto2022persistentcuplength}. 
Notably, for certain spaces, including those considered in this paper (such as the torus), the output of our algorithm is independent of the choice of representative cocycles (see \cref{rmk:independent of representative}). 
Also, developing efficient strategies for landmark selection and parameter tuning will be essential for scaling to broader applications.

 \begin{figure}
  \centering
  \includegraphics[width=1\textwidth]{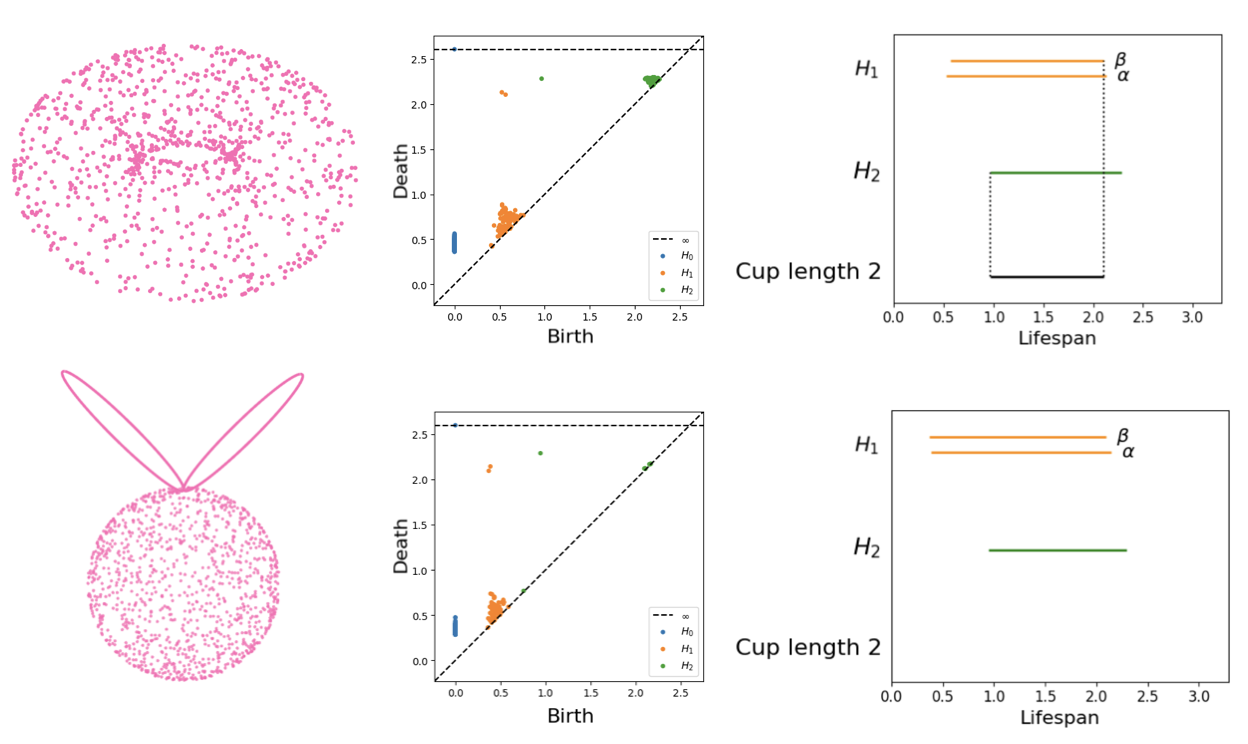}
  \caption{(Top) Torus  \(S^1\times S^1\); (Bottom) wedge sum \(S^1\vee S^2 \vee S^1 \). For each dataset: (left) point cloud, (middle) persistence diagram, and (right) cup-length computation. Persistence is computed using 150 landmarks. Both datasets exhibit one persistent connected component, two persistent 1-dimensional loops, and one persistent 2-dimensional void—features that are indistinguishable using persistent homology alone.
  The cup-length computation shows two most persistent 1-dimensional bars (orange), the most persistent 2-dimensional bar (green), as well as the resulting cup-length-2 interval (black).  Torus dataset has persistent cup-length-2 interval while \(S^1\vee S^2 \vee S^1 \) contains none.}
  \label{fig:torus-vs-bunny-persistent}
\end{figure}

\section{What is persistent cup-length, and why use it?}

Previous works (e.g., \cite{spreemann2015usingpersistenthomologyreveal, kang2021evaluating, Gardner2022}) used heuristic approaches to identify appropriate scales for detecting toroidal structure in neural data. For example, \cite{kang2021evaluating}   measured how clearly the two longest-lived 1-dimensional cycles separate from the third (expected to be short for a torus), and used this as an indicator of toroidal topology in the barcode. While effective in some settings, this heuristic is not mathematically rigorous and may fail in more complex topologies. For instance, if a circle is attached to a torus, the third persistent 1-dimensional cycle may have a comparable lifespan to the first two, potentially obscuring the torus even when it remains a dominant feature of the space.
Moreover, the usual approach of identifying two persistent 1-dimensional cycles and one persistent 2-dimensional void (for example, in \cite{Gardner2022}) is not enough to detect a torus; both torus and a wedge sum of two circles and a sphere share the same homology (see \cref{fig:torus-vs-bunny-persistent}). Consequently, to rigorously detect toroidal structure, one must go beyond persistent homology. 
The persistent cup-length implemented in this work enables us to distinguish between certain shapes that are homologically identical but topologically distinct—such as a torus and the wedge sum of two circles with a sphere—providing a mathematically principled method for detecting toroidal topology in data. 

\paragraph{Persistent cup-length.}
We build on the framework introduced in \cite{contessoto2022persistentcuplength}, which defines the \emph{persistent cup-length} function based on how cup products evolve across a filtration. This invariant captures the maximal number of cohomology classes of positive degree whose cup product remains nontrivial over a given interval, providing a computable and stable descriptor of higher-order topological structure.

We apply this construction to the Vietoris--Rips filtration of a finite metric space. Given a point cloud \(X\), the filtration \(\{\VR(X, \epsilon)\}_{\epsilon \geq 0}\) consists of simplicial complexes where each \(\VR(X, \epsilon)\) includes all simplices with vertices pairwise within distance \(\epsilon\). We denote the persistent cup-length of this filtration by \(\cupprod(X)\). See \cref{subsub:persistent cup-length} for a formal definition and discussion of its properties.

Our pipeline begins by computing the persistence diagram (or barcode; see \cite{zomorodian2004computing}) of persistent cohomology, which encodes the birth and death of cohomology classes. Each interval (or bar) in the diagram represents a cohomological feature that persists over a range of scales. We annotate each bar with a representative cocycle and compute cup products among them to build the cup-length diagram, from which we recover the cup-length function utilizing \cite[Theorem 1]{contessoto2022persistentcuplength}:
\begin{center}
    persistence diagram $\to$ \boxed{\text{cup product of annotated bars}} $\to$ cup-length diagram/function.
\end{center}
A self-contained summary of the relevant definitions and constructions is provided in \ref{sec:theory}.

In this paper, we focus on the intermediate step of computing cup products among annotated bars, which we refer to as the \emph{cup-length computation}, rather than analyzing the cup-length diagram or function which obscure information on the specific contributions of individual bars and cocycles that result in cup lengths.


\section{Algorithm}\label{sec:algorithm}
\subsection{Main algorithm}

In \cite{contessoto2022persistentcuplength}, the persistent cup-length algorithm was shown to operate in polynomial time. 
We have implemented two versions of the algorithm: one for general simplex-wise filtrations and another tailored to Vietoris–Rips filtrations computed via Ripser \cite{Bauer2021Ripser}. 
We focus here on the latter. The details of both are presented in \ref{sec:alg-main}.
 
Computing persistent cohomology on large datasets can be computationally expensive. 
We mitigate this by employing landmark-based subsampling \citep{Bauer2021Ripser}. 
Densely selected landmarks allow us to approximate the topological features of the full dataset with only a small controllable error.
This approximation is justified by stability results: if the landmarks are $\epsilon$-dense, then the persistence diagram changes by at most $\epsilon/2$ in bottleneck distance \cite{cohen2007stability}, and the persistent cup-length invariant changes by at most $\epsilon/2$ in erosion distance \cite[Theorem 2]{contessoto2022persistentcuplength}.

Our implementation operates on a distance-matrix representation of the input data and uses coefficients in the field $\mathbb{Z}/2\mathbb{Z}$. 
We employ Ripser to compute barcodes along with a set of representative cocycles.
The core components of our algorithm are:
\begin{enumerate}
    \item computing cup products at the cochain level \cite[Algorithm 1]{contessoto2022persistentcuplength}, and
    \item solving the coboundary problem across filtration scales—namely, determining the earliest time (moving from larger to smaller times) at which a given cocycle becomes a coboundary.
\end{enumerate}
The second step critically depends on the coboundary matrix and requires precise knowledge of the filtration times for each simplex. 
However, Ripser does not directly output filtration times or the full coboundary matrix needed for cup product computations, particularly in the landmark-based setting. Recovering this information requires reconstructing the simplices, their filtration indices, and the landmark set—a nontrivial step in our implementation. We addressed this using the combinatorial number system and utilities from the DREiMac library \cite{Perea2023}.

In addition, we revised the persistent cup-length algorithm from \cite{contessoto2022persistentcuplength} to correct several implementation-level inaccuracies and improve computational reliability. Key fixes include correcting the birth and death times used to initialize the output matrix, properly carrying over previously computed information between iterations, and accurately updating cup-length values when new nontrivial cup products are found. Further details are provided in \ref{sec:alg-main}.

\subsection{Algorithm validation and its properties}\label{sec:algorithm-validation}

We validated our implementation on synthetic data with both continuous and simplexwise filtrations. We discuss a few examples in this section. The standard procedure for applying our persistent cup-length algorithm is as follows. 
First, compute a distance matrix from the input point cloud (we use the Euclidean distance in all examples discussed below except for Section \ref{donut-vs-mickey}). Then, compute a Vietoris–Rips filtration with landmark-based subsampling to generate persistence diagrams and representative cocycles using Ripser \citep{ctralie2018ripser}. Finally, apply our algorithm to compute the persistent cup-length.

Throughout all experiments in this work, we verified that any cup-length-2 interval, when present, was generated by two independent 1-cocycles. This condition is necessary to apply Theorem \ref{thm:torus} and conclude the presence of a toroidal component.

\subsubsection{Doughnut versus Mickey Mouse}\label{donut-vs-mickey}

We model the torus as $T^2 = S^1 \times S^1$, representing each point by angles $(\theta, \phi) \in [0, 2\pi)^2$. Here, $S^1$ denotes the unit circle, and the intrinsic distance on the torus is defined as the product of circular geodesic distances on each factor:

$$
d((\theta_1, \phi_1), (\theta_2, \phi_2)) = \sqrt{d_{S^1}(\theta_1, \theta_2)^2 + d_{S^1}(\phi_1, \phi_2)^2},
$$

where $d_{S^1}(\alpha, \beta) = \min(|\alpha - \beta|, 2\pi - |\alpha - \beta|)$ denotes the shortest arc distance on the circle. We sample 1,000 points uniformly from the torus $S^1 \times S^1$ (see \cref{fig:torus-vs-bunny-persistent}, top left), and construct a distance matrix using the defined intrinsic metric.

We construct a wedge sum space $S^1 \vee S^2 \vee S^1$, where each component is endowed with its standard intrinsic metric (arc length on $S^1$, geodesic on $S^2$), and all are glued at a common basepoint $x_0$. The distance between points is defined as follows. Within the same component, use the intrinsic geodesic distance. Between components, for $x \in A$, $y \in B$, define

$$
d(x, y) = d_A(x, x_0) + d_B(y, x_0).
$$

This defines a path metric on the glued space that treats all transitions between components as routed through the basepoint.  We sample 1,200 points uniformly from the sphere and 400 points from each circle.

We apply the Vietoris-Rips filtration method with 150 landmarks to generate persistence diagrams and compute the cup product of annotated bars (see \cref{fig:torus-vs-bunny-persistent}, middle and right panels). Our results demonstrate that although persistent homology alone cannot distinguish the two spaces—both exhibit one persistent connected component, two persistent 1-dimensional loops, and one persistent 2-dimensional void—our algorithm successfully identifies a persistent cup-length-2 interval for the torus, whereas $S^1 \vee S^2 \vee S^1$ exhibits no cup-length greater than 1. Moreover, the cup-length-2 interval was generated by two independent 1-cocyles in the torus. This detection confirms toroidal structure according to Theorem \ref{thm:torus}.

\subsubsection{Stability under wedge sum with circles and spheres}

Persistent cup-length algorithm can detect toroidal structure even in the presence of additional circles or spheres. Specifically, by Proposition 11 in \cite{contessoto2022persistentcuplength}, the cup-length function of a torus $T^2$  with a number of circles attached is the same as the cup-length function of the torus itself since $\cupprod(S^1) = 1$. In other words, $\cupprod(T^2\vee S^1 \vee \dots \vee S^1 ) = \cupprod(T^2)$. Figure \ref{fig:torus-wedge-circles} confirms this result in a special case where a single circle is attached to the torus at one point. We reliably detect cup-length-2 interval, generated by two independent 1-cocyles,  with or without the presence of circles. Note that the same result holds if we attach spheres $S^2$ to the torus since $\cupprod(S^2) = 1$. Thus,  $\cupprod(T^2\vee S^2 \vee \dots \vee S^2) = \cupprod(T^2)$. 

 \begin{figure}
  \centering
  \includegraphics[width=1\textwidth]{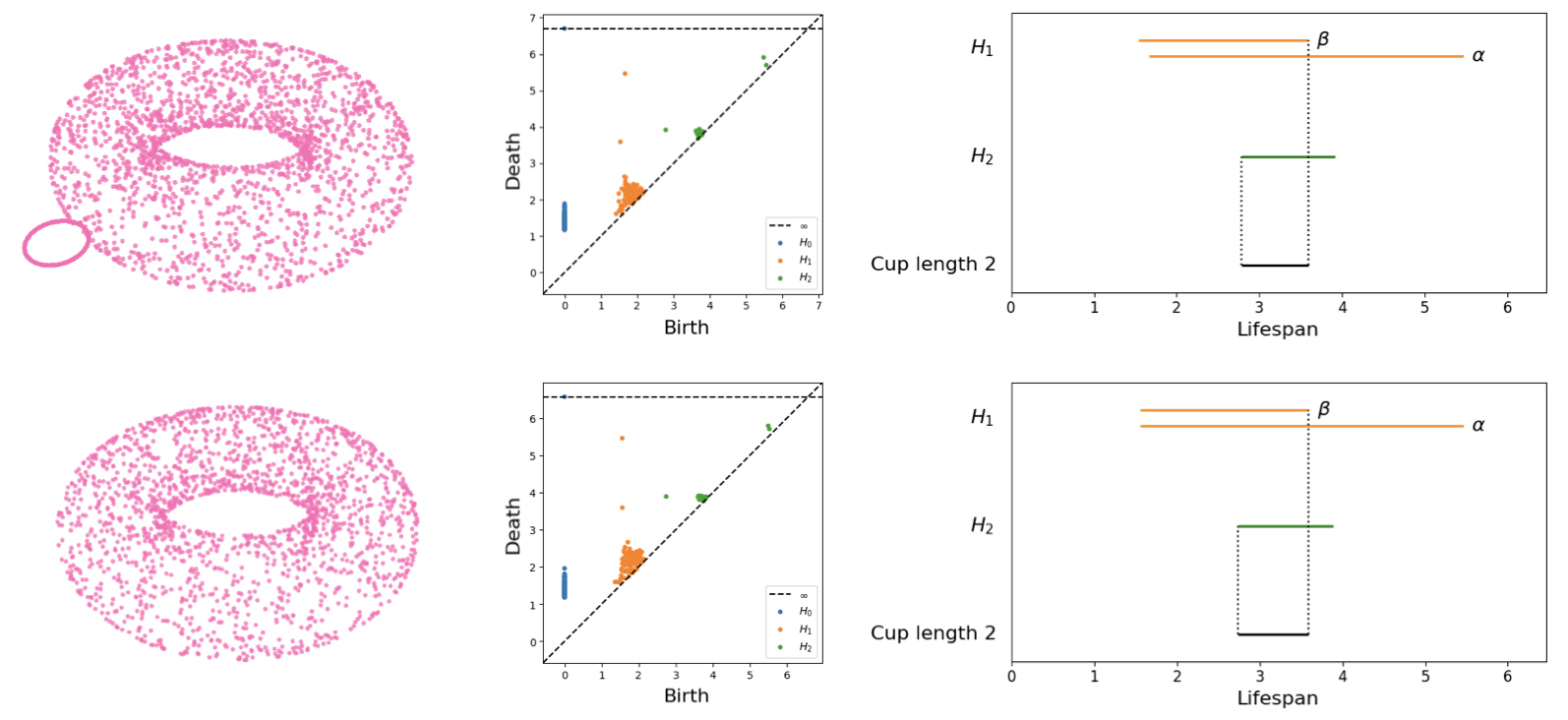}
  \caption{(Top) $T^2 \vee S^1$; (Bottom) $T^2$: (left) point cloud (2,300 points for $T^2 \vee S^1$, 2,000 points for $T^2$), (middle) persistence diagram computed using 150 landmarks, and (right) cup-length computation.  The cup-length computation shows two most persistent 1-dimensional bars (orange), the most persistent 2-dimensional bar (green), as well as the resulting cup-length-2 interval (black). The cup-length 2 is unaffected by the presence of the attached circles.}
  \label{fig:torus-wedge-circles}
\end{figure}

\subsubsection{Cup-length doesn't always come from the most persistent 2-bar}

Previous works have relied on heuristics—such as identifying two most persistent 1-dimensional loops and one most persistent 2-dimensional void \cite{Gardner2022}—to suggest the presence of toroidal structure. However, these heuristics are limited and can be misleading. Our persistent cup-length algorithm offers a more nuanced view. For example, consider the space formed by attaching a sphere to a torus at a single point (see Figure \ref{fig:torus-wedge-big-sphere}). In this example, the most persistent 2-dimensional feature corresponds to the spherical void, which does not contribute to any nontrivial cup product indicative of a torus. The true toroidal structure arises from the two most persistent 1-dimensional loops and the second most persistent 2-dimensional void, as shown in the right panel of Figure \ref{fig:torus-wedge-big-sphere}.

 \begin{figure}
  \centering
  \includegraphics[width=1\textwidth]{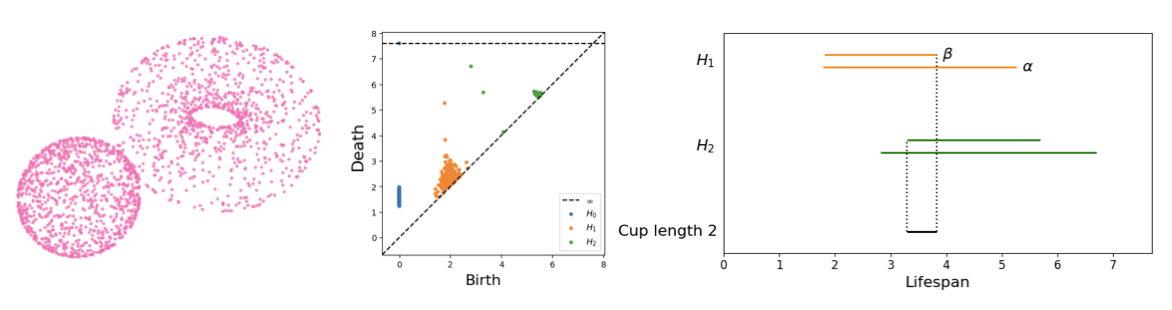}
  \caption{$T^2 \vee S^2$: point cloud of 2,000 points (left), persistence diagram (middle), and cup-length computation (right). Persistence is computed using 250 landmarks. The cup-length computation shows two most persistent 1-dimensional bars (orange),  two most persistent 2-dimensional bars (green), as well as the resulting cup-length-2 interval (black). The toroidal structure arises from the two most persistent 1-dimensional bars and the second most persistent 2-dimensional bar.}
  \label{fig:torus-wedge-big-sphere}
\end{figure}

\subsubsection{Short 2-bar can still indicate toroidal structure}

In this example, we sample points from a torus but deliberately remove a “cap” — a patch of surface defined by restricting one or both of the torus’s angular coordinates to lie within a given interval. Because this defect delays the appearance of the 2-dimensional void in the Vietoris–Rips filtration, the persistence diagram by itself may not reveal that the data originally came from a torus with a missing patch. Nevertheless, our algorithm still uncovers a cup‐length-2 interval, generated by two independent 1-cocyles, which signals that the underlying topology is toroidal (see Figure \ref{fig:torus-with-cap-removed}).

 \begin{figure}
  \centering
  \includegraphics[width=1\textwidth]{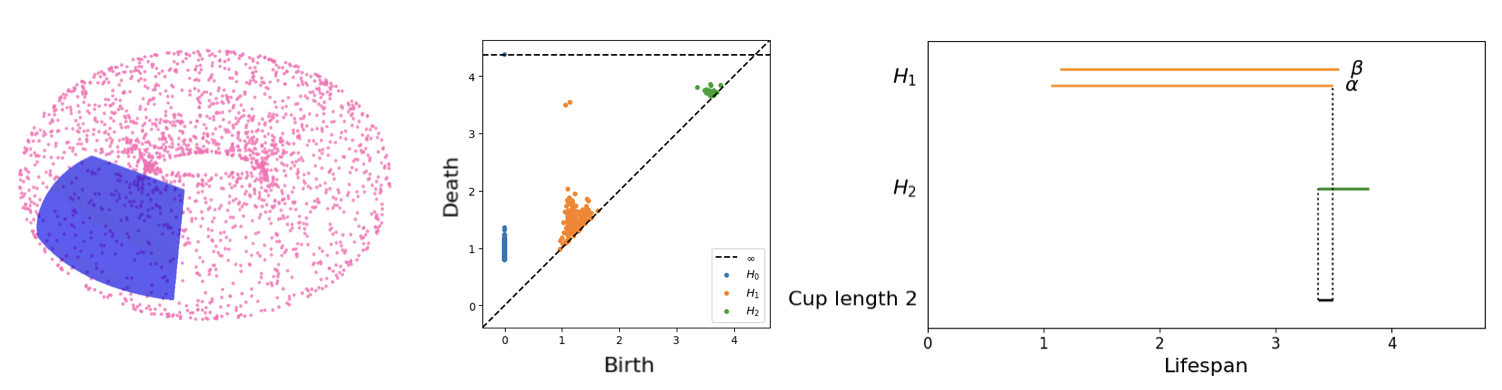}
  \caption{$T^2$ with a cap removed: point cloud of 2,000 points with the excluded region shaded blue (left), persistence diagram (middle), and cup-length computation (right). Persistence is computed using 250 landmarks. The cup-length computation shows two most persistent 1-dimensional bars (orange),  one most persistent — but still brief — 2-dimensional bar (green), as well as the resulting cup-length-2 interval (black).}
  \label{fig:torus-with-cap-removed}
\end{figure}

\subsubsection{Short cup-length 2 can still indicate toroidal topology}\label{sec:deformed-torus}

Figure \ref{fig:torus-deformed} (left) shows 2,000 points sampled from a torus whose minor radius has been locally enlarged along one generating circle. Although its persistence diagram (Figure \ref{fig:torus-deformed}, center) still contains two long-lived 1-bars and one long-lived 2-bar, the cup–length-2 interval (Figure \ref{fig:torus-deformed}, right) is very short. Hence, while the existence of a cup-length-2 interval,  generated by two independent 1-cocyles, confirms toroidal topology, its lifetime also reflects geometric deformation. In the standard torus, this interval persists over a longer period in the Vietoris–Rips filtration, but when one cycle is expanded, it appears only transiently. This nuanced perspective aids in interpreting experimental results on real data (e.g., grid-cell recordings in Section \ref{app:grid cell}).

 \begin{figure}
  \centering
  \includegraphics[width=1\textwidth]{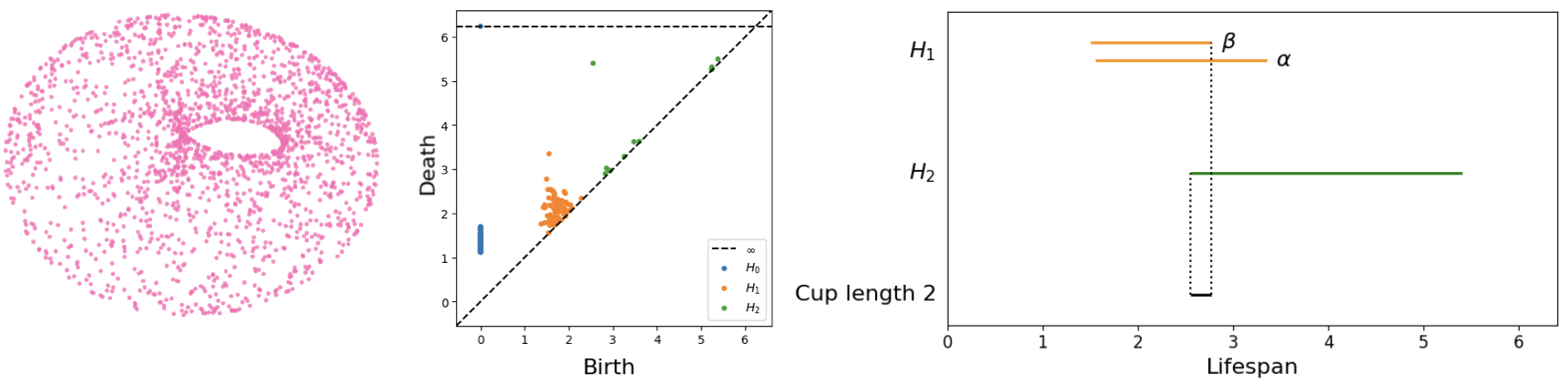}
  \caption{$T^2$ with a deformation: point cloud of 2,000 points (left), persistence diagram (middle), and cup-length computation (right). Persistence is computed using 150 landmarks. The cup-length computation shows two most persistent 1-dimensional bars (orange),  one most persistent 2-dimensional bar (green), as well as the resulting cup-length-2 interval (black).}
  \label{fig:torus-deformed}
\end{figure}

\subsubsection{Robustness to noise}

We sampled $N=2000$ points on a torus of major radius $R=5$ and minor radius $r=2$, then added i.i.d.\ Gaussian noise $\mathcal{N}(0,\sigma^2 I)$ with $\sigma\in\{0.05, 0.1, 0.2, 0.25, 0.3\}$. For each noise level and each landmark count $n_{\rm land}\in\{150,200,250\}$, we ran 5 trials to check whether our persistent cup-length algorithm can still detect a nontrivial cup-length-2 interval, recovering toroidal topology. We call a trial successful if a nontrivial cup-length-2 interval,  generated by two independent 1-cocyles, was identified. The reported detection rate in Figure \ref{fig:detection-rate} is the fraction of successful trials per $(\sigma,n_{\rm land})$ pair. At low noise levels ($\sigma\le0.2$), the algorithm achieves a 100\% detection rate for all landmark counts.  At the highest noise tested ($\sigma=0.3$), performance degrades to 0\% at $n_{\rm land}=150$, improves modestly to 20\% at $n_{\rm land}=200$, and recovers to 80\% at $n_{\rm land}=250$.  Increasing landmark density can partially compensate for high noise, while low noise is effectively handled even with relatively few landmarks.

 \begin{figure}
  \centering
  \includegraphics[width=0.4\textwidth]{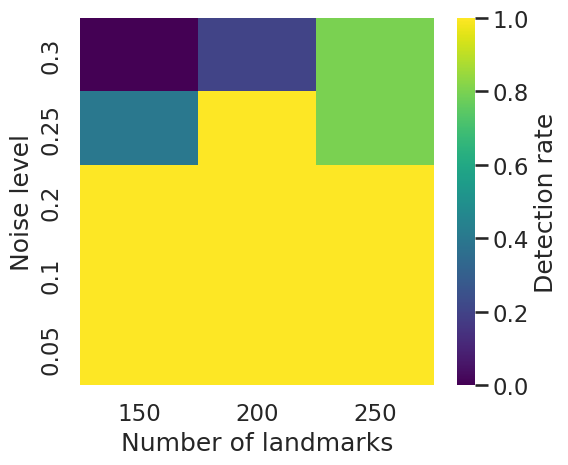}
  \caption{Detection rate as a function of noise level $\sigma$ and the number of landmarks.}
  \label{fig:detection-rate}
\end{figure}

\subsubsection{Grid cell simulations}

In this section, following the procedure outlined in \cite{kang2021evaluating}, we simulate grid cell recordings and apply  persistent cup-length to discover topological structure within the data, contrasting it with persistent homology. We simulated a single grid‐cell module with scale \(l=40\)\,cm and orientation \(\phi=0\).  For each trial, we sampled \(N\) phase‐offset vectors 
\[
b_i \;\sim\;\mathrm{Uniform}\bigl([-{\tfrac12},{\tfrac12}]^2\bigr),
\]
and evaluated the tuning curve
\[
s_{\rm grid}(x; b)\;=\;f\!\Bigl(\,\frac{1}{0.45\,l}\|\,A\,\langle A^{-1}x - b\rangle_{1/2}\|\Bigr),\]
where
\[
f(z)=\begin{cases}\frac{1+\cos(\pi z)}{2},&|z|<1,\\0,&\text{otherwise,}\end{cases}, \quad
A = l \begin{pmatrix}
\cos \phi & \cos\left(\phi + \frac{\pi}{3} \right) \\
\sin \phi & \sin\left(\phi + \frac{\pi}{3} \right)
\end{pmatrix}
\]
along a simulated random walk trajectory over \(T=1000\)\,s (sampled every \(0.2\)\,s, \(5{,}000\) points).  We zeroed out all bins where speed \(<5\)\,cm/s, normalized each neuron by its mean rate, discarded near‐zero timepoints, and geometrically subsampled to \(1{,}000\) points via farthest‐point (max–min) sampling.  Finally, we computed the Vietoris–Rips persistence diagram and used the “largest‐gap” heuristic in \(\rmH_1\), as in \cite{kang2021evaluating}, to define detection rate for persistent homology to heuristically identify a torus (see \cref{fig:success_rate_vary_grid_count}, left). We then applied the persistent cup-length algorithm to the same data; we measured the detection rate by considering whether a nontrivial cup-length-2 interval,  generated by two independent 1-cocyles, was identified (see \cref{fig:success_rate_vary_grid_count}, right).  Persistent cup-length  provides a rigorous topology check while   persistent homology provides a heuristic. We see that for both approaches,  detection rate increases with the number of grid cells, with roughly 20 simulated, idealized grid cells needed for reliable recovery of toroidal topology. We further tested how the two approaches perform as we vary the duration of the recording from 100 s to 1,000 s, in addition to varying the grid cell number. The results are shown in Figure \ref{fig:success_rates_diff_grid_cell_count}. Persistent cup-length algorithm requires roughly 100 idealized grid cells and 250 s recording for reliable topological discovery. Similar results are observed for persistent homology.


 \begin{figure}
  \centering
  \includegraphics[width=1\textwidth]{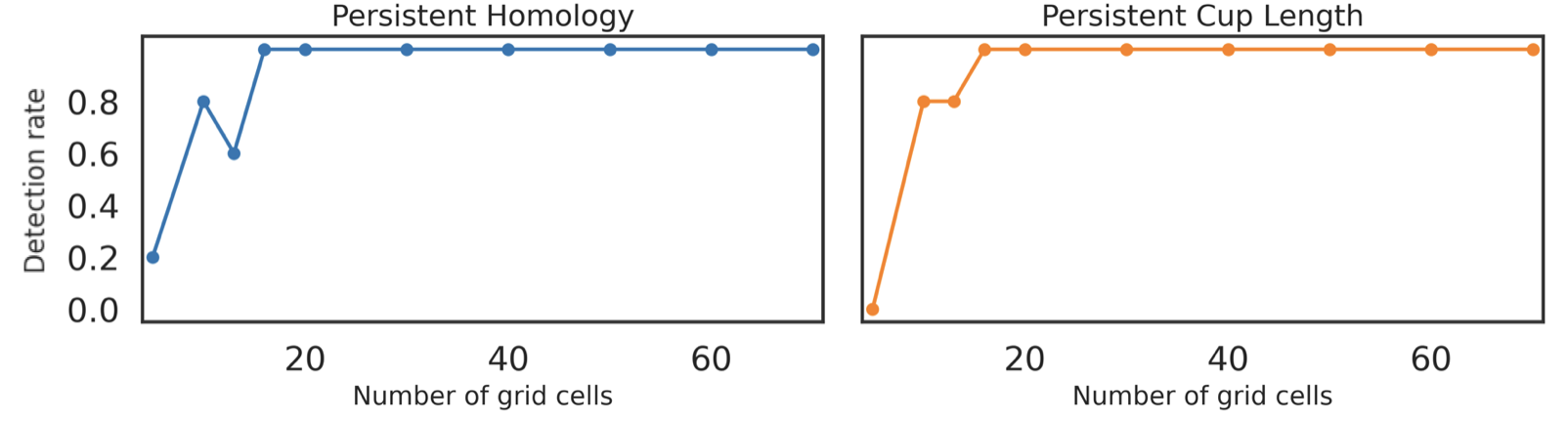}
  \caption{Detection rate as a function of the number of grid cells. (Left) \textit{Heuristic detection} via standard persistent homology—counted as “detected” whenever two distinct 
long 1-bars are found.
(Right) \textit{Rigorous detection} via our persistent cup-length algorithm—counted as “detected” only when a nontrivial cup-length-2 interval is found. In each panel, we plot the detection rate (fraction of five independent trials) that identify toroidal structure against the number of grid cells (persistence computed with 250 landmarks). Although the heuristic $\rmH_1$-count can trigger at  lower numbers of grid cells, it provides no guarantees; our cup-length test consistently uncovers \textit{true} toroidal topology once a modest grid density is reached.}
  \label{fig:success_rate_vary_grid_count}
\end{figure}

 \begin{figure}
  \centering
  \includegraphics[width=1\textwidth]{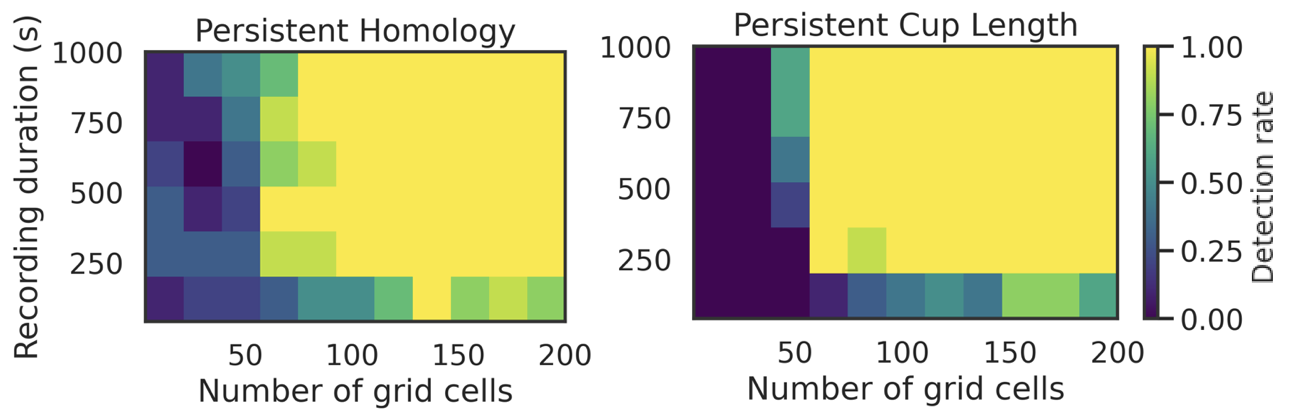}
  \caption{(Left) persistent homology; (right) persistent cup-length. Detection rates for different durations and number of grid cells. The detection of toroidal topology benefits from longer
recording durations and more grid cells. 250 landmarks were used to compute persistence; the results were averaged over 5 trials.}
  \label{fig:success_rates_diff_grid_cell_count}
\end{figure}

\section{Persistent cup-length detects toroidal structures}\label{sec:toroidal-theory}

Recall that we say that a space \( X \) has a toroidal component if its cohomology ring contains a subring isomorphic to the cohomology ring of \( T^2 \).  
In the case where $X$ is a closed, connected, orientable surface, this algebraic condition has a precise topological meaning: by the classification of surfaces, it implies that $X$ is a connected sum of tori. The following theorem formalizes this connection, providing a concrete characterization of toroidal structure in terms of cup products.

\begin{theorem}\label{thm:torus}
Let \(X\) be a closed, connected, orientable surface. Suppose there exist two linearly independent classes \(\alpha, \beta \in \mathrm{H}^1(X; \mathbb{Z}/2\mathbb{Z})\) with $\alpha \smile \beta \neq 0$;  note that orientability ensures $\alpha \smile \alpha = \beta \smile \beta = 0$.
 
Then \(X\) contains a \emph{toroidal component}; that is, in a connected sum decomposition of \(X\) 
there exists a subspace that is homotopy equivalent to the 2-torus \(T^2\).
\end{theorem}

\begin{proof}
According to \cite[Section 1.2]{hatcher2000}, any closed, connected surface is homeomorphic to either a sphere, a connected sum of tori, or a connected sum of projective planes.  
Since \( X \) has nontrivial cup products, it cannot be a sphere.  
Moreover, because \( X \) is orientable, it cannot be a connected sum of projective planes.  
Thus, by the classification theorem, \( X \) must be of the form  
\[
X \cong T^2 \# \cdots \# T^2,
\]
where the number of \( T^2 \) summands equals the genus \( g \) of \( X \).  
Since the cohomology ring is nontrivial, we have \( g \geq 1 \), ensuring that at least one torus \( T^2 \) appears in the connected sum decomposition. \qedhere
 
\end{proof}

\section{Revealing toroidal organization in grid cell populations}
\label{app:grid cell}
 
To effectively navigate in complex environments, we need to be able to localize ourselves in space. In 2014, May-Britt Moser \& Edvard Moser discovered grid cells, which are neurons that act as sensors for position. Grid cells fire in a hexagonal pattern of locations, and are organized in modules that collectively form a population code for the animal’s position. 

Recently, \cite{Gardner2022} used persistent homology to show that grid cell population activity spans a manifold consistent with toroidal topology, as expected in a two-dimensional continuous attractor network (CAN) model. To draw this conclusion, \cite{Gardner2022} primarily relied on the heuristic observation that, in 24 out of the 27 grid modules analyzed, four long-lived bars -- one connected component, two 1-dimensional loops, and one 2-dimensional void -- were observed. However, as demonstrated in \cref{fig:torus-vs-bunny-persistent}, such approach alone does not provide conclusive evidence that the underlying topology is indeed toroidal.  

Analyzing  the population activity from high-density recordings from grid cell modules with persistent cup-length can provide a rigorous test of the underlying topology. Specifically, the detection of a persistent cup-length of 2 by our algorithm indicates a stable toroidal structure throughout the filtration, which provides a stronger evidence for CAN theories \citep{Gardner2022}.  
By employing the persistent cup-length algorithm, we reliably confirm the toroidal topology of grid cell population activity in 17 out of 27 grid modules analyzed in \cite{Gardner2022} spanning varying behavioral conditions and different environmental contexts. The invariance of the toroidal topology across  environments and brain states suggests an intrinsic, robust network architecture \cite{Gardner2022}. 

\subsection{Data and its preprocessing} To perform our study, we use the data\footnote{\href{https://figshare.com/articles/dataset/Toroidal_topology_of_population_activity_in_grid_cells/16764508}{https://figshare.com/articles/dataset/Toroidal\_topology\_of\_population\_activity\_in\_grid\_cells/16764508}} 
from \cite{Gardner2022} on three experimentally naive male Long Evans rats (Rats Q, R, and S, 300–500 g at time of implantation). The data consisted of  many hundreds of simultaneously recorded grid cells. The authors recorded extracellular spikes of  thousands of single units in layers II and III of the MEC–parasubiculum region in the freely moving rats with unilateral or bilateral implants. During recordings, the rats were engaged in foraging behaviour in a square open-field (OF) enclosure or on an elevated track (WW), or they slept in a small resting box. We used the same data preprocessing strategy\footnote{ \href{https://github.com/erikher/GridCellTorus}{https://github.com/erikher/GridCellTorus}} as in the original study: Spike times were binned, smoothed, and $z$-scored to account for variability of average firing rates across cells and to decrease computational complexity. Then principal component analysis was applied with 6 components and the point cloud was downsampled to create a $1,200\times 1,200$ distance matrix. Further details on the data collection procedure and preprocessing are available in \cite{Gardner2022}. 

\subsection{Visualization}
 To visualize the population activity of grid cells, we created a 3D representation of the firing rates of each module of grid cells  by applying uniform manifold approximation and projection  (UMAP, \cite{2018arXivUMAP}) to the PCA-reduced representation. The hyperparameters for UMAP were specified as follows: n\textunderscore components=3, metric=“cosine”, n\textunderscore neighbors=5000, min\textunderscore dist=0.8. \cref{fig:rat_r}  (left) shows an example of the resulting 3D visualization reflecting a torus-like pattern.

\subsection{Persistent cup-length analysis}
 
We analyzed 27 distinct grid cell modules across a range of behavioral and sleep conditions, including open field (OF), wagon wheel (WW), rapid eye movement (REM) sleep, and slow-wave sleep (SWS), recorded from rats R, Q, and S. We subsampled the distance matrix for each grid module using 500 landmarks and created a Vietoris–Rips filtration using the Ripser library. We then applied the persistent cup-length algorithm  to assess the presence of a toroidal structure.   Our algorithm computed cup products between all available 1-dimensional cocycles in cohomology and identified instances where these products persisted over a range of filtrations. While we computed cup products for all 1-dimensional cocycles, in all cases where a cup-length value of 2 was detected, it arose specifically from the two most persistent 1-dimensional cycles and the most persistent 2-dimensional void.  

We detected cup-length 2 in 17 out of 27 grid modules, providing a robust confirmation of toroidal topology in those modules by Theorem \ref{thm:torus}. For example, results for one grid module from rat R are shown in \cref{fig:rat_r}, where a persistent cup-length-2 interval confirms the toroidal topology. Results for all analyzed modules are presented in Figures \ref{fig:rats1}–\ref{fig:rats9}. Each panel displays a persistence diagram alongside the corresponding cup product computation between the two most persistent 1-dimensional cocycles; the detected cup-length-2 interval, if detected, is also shown. We follow the naming convention of \cite{Gardner2022}, identifying each module by rat (R, Q, or S), module number (e.g., 1, 2, 3), and condition (OF, WW, REM, or SWS), with day numbers used to distinguish multiple sessions of the same module.

Several analyzed grid modules are of particular interest. For example, although modules Q2 in SWS (Figure \ref{fig:rats4}, middle panel) and Q2 in WW  (Figure \ref{fig:rats5}, bottom panel)  each exhibit two persistent 1-dimensional bars and one persistent 2-dimensional bar in their persistence diagrams consistent with toroidal topology, our cup length computation reveals a more nuanced picture. Specifically, the cup-length-2 interval is short, indicating that the toroidal structure, while present, persists only briefly in the filtration. This could suggest that the underlying topology is not that of a standard torus; it can instead, for example, come from a deformed torus—one that is compressed or "squashed" along one direction (see Section \ref{sec:deformed-torus} for a simulation).

In modules R1 in WW on day 1 (Figure \ref{fig:rats7}, middle panel) and R3 in OF on day 2 (Figure \ref{fig:rats7}, bottom panel), the persistence diagrams contain one persistent 0-dimensional bar, two persistent 1-dimensional bars, and one persistent 2-dimensional bar. However, the cup product computation indicates that these features do not generate a torus: the birth time of the 2-dimensional class occurs after the death time of the longest 1-dimensional class, thereby precluding a cup-length 2 structure arising from these classes.

Finally, we note that we also did not detect cup-length 2 in R1 SWS (day 2), R1 REM (day 2), R3 SWS (day 2), Q1 REM, Q2 REM, S1 WW, S1 REM, and S1 SWS modules (see Figures \ref{fig:rats6}-\ref{fig:rats9}). These failures may be attributed to low cell counts, heterogeneous module composition, or effects of our  subsampling procedure, as evidenced by noisy and less interpretable persistence diagrams. These modules will be examined in more detail in future work. 

For comparison, \cite{Gardner2022} reported a lack of consistency with toroidal topology only in module R1 during SWS, and module S1 during REM and SWS. In the latter two cases, the likely explanation is an insufficient number of recorded cells (72 cells), while the absence in R1 during SWS was attributed to heterogeneous cell composition. Notably, when analysis was restricted to the bursty (B) cell class in R1, a toroidal structure became apparent.

\begin{figure*}[!htb]
\begin{center}
\includegraphics[width=1\textwidth]{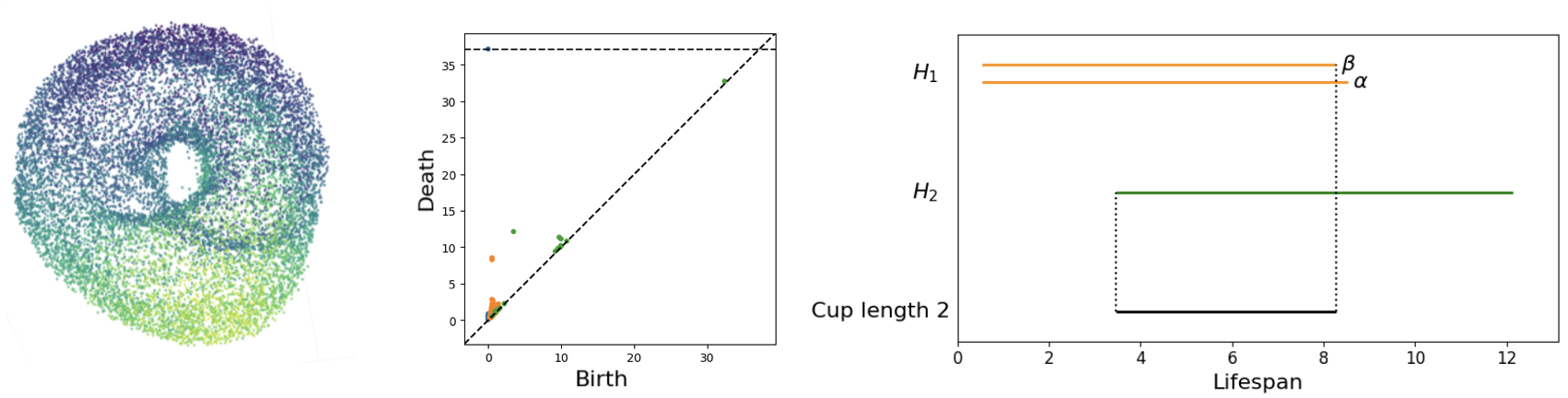}
\end{center}
\caption{ (Left) UMAP-based visualization of the population activity of a single grid module (149 grid cells, rat R, day 1, OF, module 2). Each dot represents the population state at one time point (dots colored by first principal component).  (Middle) Persistence diagram of the population activity. (Right) Cup-length computation showing two most persistent 1-dimensional bars (orange), the most persistent 2-dimensional bar (green), as well as the resulting cup-length-2 interval (black).}\label{fig:rat_r}
\label{fig1}
\end{figure*}

\section{Conclusion}\label{sec:conclusion}

We presented the first practical implementation of persistent cup-length, a cohomological invariant that overcomes key limitations of persistent homology in detecting toroidal structure. Unlike heuristic approaches, our method provides a rigorous, algebraic test for toroidal topology. Applied to both synthetic data and grid cell recordings, our method reliably detected toroidal components.  

Despite these promising results and applications, several limitations remain to be addressed in future work. Specifically, the computational complexity of the persistent cup-length algorithm, although polynomial, may be substantial for large datasets, requiring additional optimization and efficient landmark selection strategies. Additionally, the method's sensitivity to cocycle representative selection and filtration parameters necessitates careful preprocessing and parameter tuning for robust performance. At present, our implementation is restricted to computing cup-length 2; although the algorithm extends naturally in theory to higher cup lengths, overcoming computational challenges will be essential for generalizing the approach. Future studies will also explore extensions incorporating higher-order cohomological operations, such as Massey products \cite{flynnconnolly2024higherordermasseyproducts} and Steenrod squares \cite{medina2022per_st}.  

In addition, future work will investigate applying persistent cup-length to analyze dynamical regimes in recurrent neural networks (RNNs). By capturing nontrivial interactions among persistent cocycles, persistent cup-length has the potential to differentiate effectively between periodic, quasi-periodic, and chaotic dynamics within RNN hidden states. In particular, it may accurately identify toroidal attractors emerging from quasi-periodic signals \cite{gakhar2024sliding}, as seen in tasks like spatial navigation and path integration \cite{delmastro2024on,Gardner2022, NEURIPS2024_285b06e0}. Furthermore, extending this analysis to the evolution of RNN weight trajectories during training could yield deeper insights into learning dynamics and enhance model expressivity, aiding the design of neural network architectures optimized for capturing complex temporal dependencies \cite{birdal2021intrinsic,Sussillo2013OpeningTB,NEURIPS2019_5f5d4720}.

\section{Acknowledgements}

ESI thanks Jacob Zavatone-Veth and Andrew Yarmola for helpful discussions. GA acknowledges support from the Fannie and John Hertz Foundation. LZ thanks Jose Perea for helpful discussions and acknowledges support from the AMS-Simons travel grant. This material is based upon work supported by the National Science Foundation Graduate Research Fellowship Program under Grant Noch DGE 2140743. Any opinions,
findings, and conclusions or recommendations expressed in this material are those of the
author(s) and do not necessarily reflect the views of the National Science Foundation.

\bibliographystyle{plain}
\bibliography{bibliography}
 
\newpage
\appendix

\section{Technical Appendices and Supplementary Material}

\subsection{Cup-length diagram and cup-length function}\label{sec:theory}

We provide background on the (persistent) cup-length diagram and the (persistent) cup-length function introduced in \cite{contessoto2022persistentcuplength}, and along the way recall the classical notion of cup-length.

\subsubsection{Cup product} Let $\X$ be a finite simplicial complex with totally ordered vertices $\{x_1 < \dots < x_n\}$. 
For a non-negative integer $p$, a \emph{$p$-simplex} is a list of $(p+1)$ distinct vertices in increasing order. Let $\X_p \subset \X$ denote the set of all $p$-simplices. Fix a filed $\field$. The space of $p$-chains, denoted $C_p(\X)$, is the $\field$-vector space spanned by $\X_p$.

A \emph{$p$-cochain} is a linear map $\sigma: C_p(\X) \to \field$. 
The space of $p$-cochains is denoted $C^p(\X) := \operatorname{Hom}(C_p(\X), \field)$. 
For each $p$-simplex $\alpha$, the \emph{dual cochain} $\alpha^* \in C^p(\X)$ is defined by $\alpha^*(\alpha) = 1$ and $\alpha^*(\tau) = 0$ for all $\tau \neq \alpha$. 
These $p$-cosimplices form a basis for $C^p(\X)$, so any $p$-cochain can be written as a finite linear combination of them

In our implementation, we work over the field $\field = \mathbb{Z}/2\mathbb{Z}$ and encode each $p$-simplex $\alpha = [x_{i_0}, \dots, x_{i_p}]$ by its index list $[i_0, \dots, i_p]$, following the total simplex ordering used in \cite{Bauer2021Ripser}. 

Since coefficients are binary, a $p$-cochain $\sigma$ can be written as a sum of cosimplices: 
\[
\sigma = \sum_{j=1}^h (\alpha^j)^*,
\]
where each $\alpha^j = [i_0^j, \dots, i_p^j]$ is a $p$-simplex. We encode such a cochain as a list of index sequences:
\[
\left[ [i_0^1, \dots, i_p^1], \dots, [i_0^h, \dots, i_p^h] \right].
\]
The integer $h$ is referred to as the \emph{size} of the cochain $\sigma$.

Given $p$- and $q$-simplices $\alpha = [\alpha_0, \dots, \alpha_p]$ and $\beta = [\beta_0, \dots, \beta_q]$, their duals $\alpha^*$ and $\beta^*$ are $p$- and $q$-cosimplices, respectively. The \emph{cup product} $\alpha^* \smile \beta^*$ is the cochain in $C^{p+q}(\X)$ defined on any $(p+q)$-simplex $\tau = [\tau_0, \dots, \tau_{p+q}]$ by
\[
\alpha^* \smile \beta^*(\tau) := \alpha^*([\tau_0, \dots, \tau_p]) \cdot \beta^*([\tau_p, \dots, \tau_{p+q}]).
\]
Equivalently, the result is the cosimplex $[\alpha_0, \dots, \alpha_p, \beta_1, \dots, \beta_q]^*$ if $\alpha_p = \beta_0$, and zero otherwise.

More generally, the cup product of a $p$-cochain $\sigma = \sum_{j=1}^h \lambda_j (\alpha^j)^*$ and a $q$-cochain $\sigma' = \sum_{j'=1}^{h'} \mu_{j'} (\beta^{j'})^*$ is defined bilinearly as
\[
\sigma \smile \sigma' := \sum_{j=1}^h \sum_{j'=1}^{h'} \lambda_j \mu_{j'} \left((\alpha^j)^* \smile (\beta^{j'})^*\right).
\]

Algorithm 1 in \cite{contessoto2022persistentcuplength} outlines the procedure for computing the cup product of two cochains over $\mathbb{Z}/2\mathbb{Z}$.

\subsubsection{Cohomology ring and cup-length} 
\label{subsub:cohomology and cup length}
The cochain spaces $C^p(\X)$ are connected by a sequence of linear maps $\delta: C^p(\X) \to C^{p+1}(\X)$ called the \emph{coboundary operators}. 
The kernel of $\delta$ consists of \emph{cocycles}, and the image of $\delta$ consists of \emph{coboundaries}. The $p$-th cohomology group is defined as the quotient
\[
\rmH^p(\X) := \ker(\delta: C^p(\X) \to C^{p+1}(\X)) \, \big/ \, \operatorname{im}(\delta: C^{p-1}(\X) \to C^p(\X)).
\]

The cup product operation induces a bilinear map 
\[
\smile: \rmH^p(\X) \times \rmH^q(\X) \to \rmH^{p+q}(\X),
\]
which equips the total cohomology vector space 
\[
\rmH^*(\X) := \bigoplus_{p \geq 0} \rmH^p(\X)
\]
with the structure of a graded ring.
A fundamental invariant of this ring is the \emph{cup-length}, defined as the largest integer $\ell$ such that there exist cohomology classes $\eta_1, \dots, \eta_\ell$ of positive degree with 
\(
\eta_1 \smile \cdots \smile \eta_\ell \neq 0.
\) 
We denote the cup-length of $\X$ by $\cupprod(\X)$.

\subsubsection{Cup-length diagram and cup-length function}
\label{subsub:persistent cup-length}

A \emph{filtration} $\Xfunc = \{\X_t\}_{t \geq 0}$ is a nested sequence of simplicial complexes:
\[
\X_{t_1} \subseteq \X_{t_2} \subseteq \cdots \subseteq \X_{t_k}, \quad \text{for } t_1 < t_2 < \cdots < t_k.
\]
Applying the cohomology ring functor to this sequence yields a diagram of graded rings (and graded vector spaces) connected by maps induced by the inclusions:
\[
\Hfunc^*(\X_{t_1}) \leftarrow \Hfunc^*(\X_{t_2}) \leftarrow \cdots \leftarrow \Hfunc^*(\X_{t_k}),
\]
forming the \emph{persistent cohomology ring} $\Hfunc^*(\Xfunc)$. This structure captures not only how cohomology classes evolve across the filtration, but also how their cup products persist over time.

In each fixed degree \( p \), the corresponding diagram of cohomology vector spaces
\[
\Hfunc^p(\X_{t_1}) \to \Hfunc^p(\X_{t_2}) \to \cdots \to \Hfunc^p(\X_{t_k})
\]
defines the \emph{persistent cohomology} in degree \( p \), which tracks the evolution of cohomology classes across the filtration. A class is said to be \emph{born} at time $t_i$ if it appears in $\Hfunc^p(\X_{t_i})$ but not before, and it \emph{dies} at time $t_j$ if it becomes trivial in $\Hfunc^p(\X_{t_j})$ and remains so afterward. Each such class corresponds to an interval $[t_i, t_j)$ in the \emph{barcode} $\caB_p(\Xfunc)$ of degree $p$.
For each such interval $[b,d)$, we select a representative cocycle $\sigma_I \in C^p(\X_{d - \epsilon})$ (for small $\epsilon > 0$) that generates the corresponding class. 

Given a family $\bsigma = \{\sigma_I\}_{I \in \caB_{\geq 1}(\Xfunc)}$ of representative cocycles, the \emph{(persistent) cup-length diagram (accosiated to $\bsigma$)} $\mathbf{dgm}_{\bsigma}^{\smile}(\Xfunc): \Int \to \mathbb{N}$ records, for each interval $I$, the maximal number $\ell$ such that a nontrivial $\ell$-fold product $\sigma_{I_1} \smile \cdots \smile \sigma_{I_\ell}$, with $\sigma_{I_j} \in \bsigma$, persists over $I$. 
Here, we denote by $\Int$ a fixed collection of intervals of the same type (e.g., open-open, closed-open, etc.). When results apply to all types, we write intervals as $\langle a, b \rangle$ for notational simplicity.

\begin{rmk}\label{rmk:cup length diagram doesnt detect toroidal topology}
    The presence of a toroidal component in a space cannot be determined from the cup-length diagram alone. We must also know whether a given cup-length-2 interval arises from the cup product of two linearly independent 1-cocycles. This information is available from intermediate steps in our persistent cup-length algorithm.
\end{rmk}

The \emph{(persistent) cup-length function} $\cupprod(\Xfunc): \Int \to \mathbb{N}$ assigns to each interval $\langle t, s\rangle$ the cup-length of the image ring $\mathrm{Im}(\Hfunc^*(\X_s) \to \Hfunc^*(\X_t))$, interpreted as a graded subring of $\Hfunc^*(\X_t)$. This function reflects the maximum number of persistent cohomology classes that can cup to a nonzero product over the interval.
By \cite[Theorem 1]{contessoto2022persistentcuplength}, the cup-length function can be recovered from the cup-length diagram by taking the maximum over all intervals containing the given one:
for any $\langle a,b\rangle\in \Int$,
\begin{equation}\label{eq:tropical_mobuis}
    \cupprod(\Xfunc)(\langle a,b\rangle)=  \max_{\langle c,d\rangle\supseteq \langle a,b\rangle}\mathbf{dgm}_{\bsigma}^{\smile}(\Xfunc)(\langle c,d\rangle). 
\end{equation}

See \cref{fig:cup-length diagram and function} for an example plot of the cup-length diagram and cup-length function.

\begin{rmk}\label{rmk:independent of representative}
    It is clear that the cup-length function is independent of the choice of representative cocycles, whereas the cup-length diagram generally depends on this choice, as explained in \cite[Example 18]{contessoto2022persistentcuplength}. However, for the torus $T^2$, the cup-length diagram remains unchanged under different choices of representatives. 

    To see this, let $\sigma$ and $\sigma'$ be two generators of $\rmH^1(T^2)$ such that $\sigma \smile \sigma'$ generates $\rmH^2(T^2)$. The degree-one cocycles can be chosen as any of the following pairs: $\{\sigma, \sigma'\}$, $\{\sigma, \sigma+\sigma'\}$, $\{\sigma', \sigma+\sigma'\}$, or $\{\sigma+\sigma', \sigma+\sigma'\}$. In all cases, the resulting cup-length diagrams are the same. 
\end{rmk}

We recall the following property of the cup-length function.
\begin{prop}[{\cite[Proposition 11]{contessoto2022persistentcuplength}}]
    \label{prop11}
If $\Xfunc,\Yfunc:(\mathbb{R},\leq)\to\Top$ are two persistent spaces. Then:
\begin{itemize}
    \item $\cupprod\left(\Xfunc\times\Yfunc\right)=\cupprod(\Xfunc)+\cupprod(\Yfunc),\text{  }$
    \item $\cupprod\left(\Xfunc\amalg\Yfunc\right)=\max\{\cupprod(\Xfunc),\cupprod(\Yfunc)\},\text{ and }$
    \item $\cupprod\left(\Xfunc\vee\Yfunc\right)=\max\{\cupprod(\Xfunc),\cupprod(\Yfunc)\}$.
\end{itemize}
Here $\times,\amalg$ and $\vee$ denote point-wise product, disjoint union, and wedge sum, respectively.
\end{prop}

\begin{figure}
    \centering
    \includegraphics[width=1\linewidth]{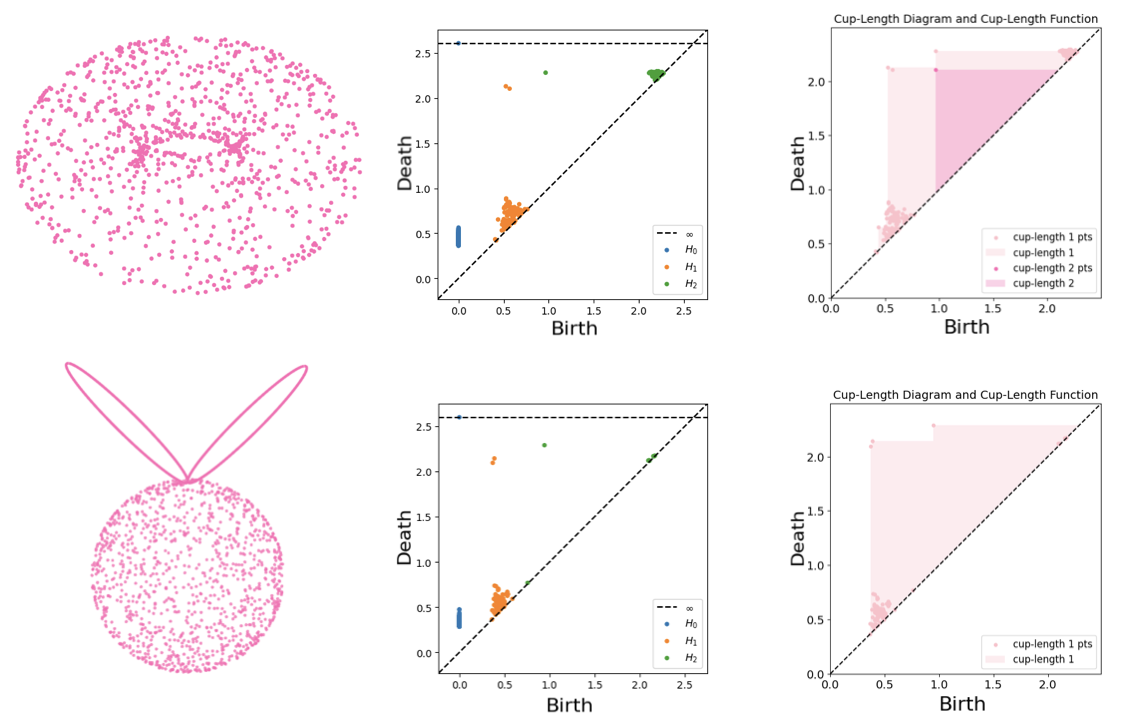}
    \caption{
    Computation of the cup-length diagram and function for the point clouds in \cref{fig:torus-vs-bunny-persistent}. The rightmost column shows a joint visualization of the cup-length diagram (colored points) and the cup-length function (shaded background), where darker colors indicate higher cup-length values. Cup-length-1 points, shown in light pink, correspond to 1- and 2-dimensional  features in the persistence diagram, shown in the middle. The cup-length-2 point, shown in dark pink, represents the cup-length-2 interval detected by our algorithm.     }
    \label{fig:cup-length diagram and function}
\end{figure}

\subsection{Algorithm's theory and implementation}

\subsubsection{Coefficient choices}

While our theorem is stated over \(\mathbb{Z}\) for generality, we compute with \(\mathbb{Z}/2\mathbb{Z}\) coefficients in practice. This choice is standard in computational persistent homology due to several technical advantages. Specifically, computations over \(\mathbb{Z}/2\mathbb{Z}\) avoid complications arising from torsion elements and simplify algebraic operations, resulting in efficient algorithms. 

The Universal Coefficient Theorem (see Section 3.A in \cite{hatcher2000}) justifies this practice, stating that cohomology groups with field coefficients (in particular, \(\mathbb{Z}/2\mathbb{Z}\)) completely capture the rank information of integral cohomology, although torsion information is lost. 

For our application—detecting toroidal structure via nonzero cup products in \(\mathrm{H}^2(X)\)—this simplification is justified because orientability of the simplicial complexes ensures the existence of globally consistent cohomology classes. Hence, the nontriviality of cup products over \(\mathbb{Z}/2\mathbb{Z}\) remains indicative of genuine topological features, such as the presence of toroidal components.

\subsubsection{Algorithm modification}
\label{sec:alg-main}

As mentioned in the main body, we revised the persistent cup-length algorithm from \cite{contessoto2022persistentcuplength} to correct several implementation-level inaccuracies and improve computational reliability. Here, we provide the details.

\cref{alg:cup_product}, as in \cite[Algorithm 1]{contessoto2022persistentcuplength}, computes the cup product of two cochains at the cochain level. For completeness, we restate it here without modification, as it serves as a foundational component of the persistent cup-length pipeline.

\begin{algorithm}[H] \label{algorithm_cup_product}
\SetKwData{Left}{left}\SetKwData{This}{this}\SetKwData{Up}{up}
\SetKwInOut{Input}{Input}\SetKwInOut{Output}{Output}
\Input{Two cochains $\sigma_1$ and $\sigma_2$, and the simplicial complex $\X$.}
\Output{The cup product $\sigma=\sigma_1\smile \sigma_2$, at cochain level.}
\BlankLine
$\sigma\gets [\,]$\;
\If{$\dim(\sigma_1)+\dim(\sigma_2)\leq \dim(\X)$}
    {\For
    {$i \leq \mathrm{size}(\sigma_1)$ and $j \leq \mathrm{size}(\sigma_2)$ }{
            $a\gets \sigma_1(i)$ and $b\gets \sigma_2(j)$\;
            \If{$a[\mathrm{end}]== b[\mathrm{first}]$}{
                $c \gets a.\mathrm{append}(b[\mathrm{second}:\mathrm{end}])$\;
                \If
                {$c\in \X_{\dim(\sigma_1)+\dim(\sigma_2)}$ \label{line:last-if-cup}}{
                     Append $c$ to $\sigma$\;
                     }
                }
            }
    }
\Return $\sigma$.
\caption{$\mathrm{CupProduct}(\sigma_1,\sigma_2,\X)$.}
\label{alg:cup_product}
\end{algorithm}

The persistent cup-length algorithm, summarized in \cref{alg:main} in the case of a simplex-wise filtration, computes the \emph{persistent cup-length matrix} $A_\ell$ for a given input filtration.  
This matrix is indexed by birth times (rows) and death times (columns), with each entry representing the cup-length value (no larger than $\ell$) corresponding to a specific birth–death interval.

Below, we summarize our corrections and modification of \cite[Algorithm 2]{contessoto2022persistentcuplength}, where all line numbers refer to \cref{alg:main}.

\paragraph{Refinements to the simplex-wise version of the alogithm.} Our modified version of \cite[Algorithm 2]{contessoto2022persistentcuplength}, presented in \cref{alg:main}, incorporates several key improvements:

\begin{itemize}
\item For clarity, \cref{alg:main} assumes simplex-wise filtrations with finite death times. However, our practical implementation also addresses scenarios involving infinite death times.
\item \textbf{Lines 1–2:} The original definitions of $\mathrm{b_{time}}$ and $\mathrm{d_{time}}$ were tailored specifically to Vietoris-Rips filtrations. We generalized these definitions by tracking all filtration indices, allowing the algorithm to handle general simplex-wise filtrations.
\item \textbf{Lines 5 and 32:} We introduced a data structure previously missing from the original algorithm to store intermediate $\ell$-fold cup products computed throughout iterations of the outer “while” loop.
\item \textbf{Lines 6–9:} We corrected an indexing error present in the original algorithm. Instead of iterating over pairs $(b_i, b_j)$, we iterate correctly over pairs of birth and death times, updating the corresponding entry in $A_1$ to reflect a cup length of 1 when the pair is found in $B_1$.
\item \textbf{Line 13:} We added the initialization step for $A_{\ell+1}$, which was omitted from the original outer “while” loop.
\item \textbf{Line 17:} We included an explicit check for triviality of the vector representation of the resulting cup product. If this vector is trivial, we proceed immediately to the next pair of barcodes, optimizing computational efficiency.
\item \textbf{Lines 19–21:} We implemented a necessary coboundary condition check at filtration level $d_{\min}$ prior to iterating backward through earlier birth times. 
\item \textbf{Line 23:} To ensure accurate  cup length interval computation, we halt the backward iteration upon reaching the smallest possible birth time.
\item \textbf{Line 27:} If the resulting cup product becomes trivial at a certain filtration time during backward iteration, we record this point as the left endpoint of the cup-length interval and terminate the search.
\item \textbf{Line 31:} We corrected the update rule to $A_{\ell+1}(i,\min\{j_1,j_2\}) = \ell + 1$ (vs. $\ell$ in the original algorithm) when a nontrivial cup product is found. 
\end{itemize}

\paragraph{Refinements to the Vietoris–Rips version of the alogithm.} We adapt \cref{alg:main} specifically for the Vietoris–Rips filtration. Since this version is structurally similar to the simplex-wise case, we omit its pseudocode. In addition to the improvements listed above, we introduce the following modifications specific to the Vietoris–Rips setting:
\begin{itemize}
    \item We incorporated bar thresholding to facilitate the selection of persistent features with longer lifespans as needed.
    \item \textbf{Line 1}: For the purpose of detecting toroidal structure, where only dimension-1 and dimension-2 bars are used, we initialize birth times for dimension 2 only.
    \item \textbf{Line 14}: We avoid redundant computations by not repeating both $\alpha \smile \beta$ and $\beta \smile \alpha$. 
    \item \textbf{Lines 19 and 22}: Our approach modifies this step by \textbf{dynamically reconstructing the coboundary matrix}, rather than using the originally proposed reduced-matrix formulation. This adaptation improves compatibility with existing TDA software—particularly Ripser, which we utilize—as these packages are optimized for computational and space efficiency and do not store intermediate reduced matrices.
\end{itemize}

\begin{algorithm}
\SetKwData{Left}{left}\SetKwData{This}{this}\SetKwData{Up}{up}
\SetKwInOut{Input}{Input}\SetKwInOut{Output}{Output}
\Input{ A dimension bound $k$, the ordered list of cosimplices $S^*$ from dimension \textcolor{red}{$0$} to \textcolor{red}{$k$}, \textcolor{red}{coboundary matrix $R$ from dimension $0$ to $k$}, and barcodes (annotated by representative cocycles) from dimension $1$ to $k$: $\mathcal{B}_{[1,k]}=\{(b_\sigma,d_\sigma,\sigma)\}_{\sigma\in \bsigma}$, where each $ \sigma$ is a representative cocycle for the bar $(b_\sigma,d_\sigma)$ and $\{ \sigma_1,\dots, \sigma_{q_1}\}$ is ordered first in the increasing order of the death time and then in the increasing order of the birth time.} 
\Output{ A cup-length matrix $A_\ell$} 
\BlankLine
$\mathrm{b\_time}\gets \textcolor{red}{[0, 1, 2, ..., |S^*| - 1]}$\;
$\mathrm{d\_time}\gets \textcolor{red}{[0, 1, 2, ..., |S^*| - 1]}$\;
$ m_\udim,\, \ell,\, B_1 \gets \mathrm{card}(S^*),\, 1,\,\mathcal{B}_{[1,\udim]}$\;
$A_0=A_1\gets \mathrm{zeros}(\mathrm{card}(\mathrm{b\_time}), \mathrm{card}(\mathrm{d\_time}))$\;
$ \textcolor{red}{B = [[], B_1]}$ \tcp*{each $B_\ell$ stores the $\ell$-fold cup products}

\For{$i = 0,\dots,\card(\mathrm{b\_time})-1$}
{
  \For{$j = 0,\dots,\card(\mathrm{d\_time})-1$}{
    \If{$(\mathrm{b\_time}[i],\mathrm{d\_time}[j])\in B_1$}{
      $A_1[i,j]\gets 1$
    }
  }
}
$ \textcolor{red}{A = [A_0, A_1]}$\;

\While(\tcp*[f]{$O(\udim)$}){$A[\ell-1]\neq A[ \ell ]$ and $l\leq k-1$ \label{line:while}}
{
    $B_{\ell+1}=[]$\; 

    $\textcolor{red}{\text{Append a copy of } A[\ell] \text{ to } A} $ \tcp*{Initialize $A_{\ell+1}$ as $A_\ell$}
    
    \For(\tcp*[f]{$O( q_1\cdot q_{k-1})$})
    {
        $(b_{i_1},d_{j_1}, \sigma_1)\in B_1$ \text{and} $(b_{i_2},d_{j_2}, \sigma_2)\in B[\ell]$ \label{line:while-for}
    }{
        $\sigma\gets \mathrm{CupProduct}( \sigma_1, \sigma_2,S^*)$ \tcp*{$O( m_\udim ^2\cdot c_\udim)$, \cref{alg:cup_product}}
        $y\gets$ the vector representation of $\sigma$ in $S^*$\;
        \If{$\textcolor{red}{y\neq 0}$}
        { 
        $d_{\min} \gets \mathrm{d\_time}[\min(j_1, j_2)]$\;
            \If(\tcp*[f]{$O(d^3_{\min} )$})
            {
                $\textcolor{blue}{\exists x \text{ s.t. } R[:d_{\min}, :d_{\min}]x = y[:d_{\min}]}$
            }{
                $\textcolor{red}{i\gets \max\{i': \mathrm{b\_time}[i']\leq d_{\min}\}}$\;
                $\textcolor{red}{s_i\gets }$ \textcolor{red}{number of simplices  with birth time } $\textcolor{red}{\leq \mathrm{b\_time}[i]}$
            
            \While(\tcp*[f]{$O(s^3_{i} )$})
            {
                $\textcolor{blue}{\exists x \text{ s.t. } R[:s_i, :s_i]x = y[:s_i]}$ \label{line:while-while}
            }{
                \If{$\textcolor{red}{i = 0}$}
                {
                    break
                }
                $i\gets i-1$\;
                $s_i\gets $ number of simplices with birth time $\leq \mathrm{b\_time}[i]$\;
                \If{$\textcolor{red}{y[:s_i] == 0}$}
                {
                    break
                }
            }
            \If{$\mathrm{b\_time}[i] < d_{\min}$}
            {
                Append $\left(\mathrm{b\_time}[i],\,d_{\min},\,\sigma\right)$ to $B_{\ell+1}$\;
                $A[\ell+1](i,\min\{j_1,j_2\})\gets \textcolor{red}{\ell + 1}$ \tcp*{Update $A_{\ell+1}$}
     
                $ \textcolor{red}{\text{Append } B_{\ell+1} \text{ to } B}$
            }} 
        }
    }
    $\ell\gets \ell+1$\; 
}
\Return~$A_\ell$.
\caption{Modified and improved version of \cite[Algorithm 2]{contessoto2022persistentcuplength}. The algorithm computes persistent cup-length matrix of a simplex-wise filtration.
Red and blue text highlight corrections and improvements, where blue text specifically marks the key step: checking whether a cocycle becomes a coboundary at a given filtration level. As in \cite{contessoto2022persistentcuplength}, we use the following notation when estimating the complexity of the algorithm: let $m_k$ denote the number of positive-dimensional simplices in the $(k+1)$-skeleton of the filtration; let $c_k$ represent the cost of checking whether a simplex is alive at a given filtration parameter; and $q_{k-1}$ is the maximum cardinality of $\ell$-fold products of elements from $\caB_{[1,k]}$, for $1 \leq \ell \leq k-1$. In particular, $q_1$ is the cardinality of $\caB_{[1,k]}$. 
In practice, since computations are over $\mathbb{Z}/2\mathbb{Z}$ and the coboundary matrix is sparse, Lines 19 and 22 are much faster than their worst-case cubic complexity.
Also note that on line 14, $i_1$ and $i_2$ index into $\text{b\_time}$, while $j_1$ and $j_2$ index into  $\text{d\_time}$.
}
\label{alg:main}
\end{algorithm}
 
\subsection{Revealing toroidal organization in grid cell populations: persistent cup-length analysis}\label{sec:grid-cell-analysis}
 
 Results for all analyzed grid cell modules are presented in Figures \ref{fig:rats1}–\ref{fig:rats9}. Each panel displays a persistence diagram alongside the corresponding cup product computation between the two most persistent 1-dimensional cocycles; the detected cup-length-2 interval, if detected, is also shown. We follow the naming convention of \cite{Gardner2022}, identifying each module by rat (R, Q, or S), module number (e.g., 1, 2, 3), and condition (OF, WW, REM, or SWS), with day numbers used to distinguish multiple sessions of the same module.


 \begin{figure}
  \centering
  \includegraphics[width=1\textwidth]{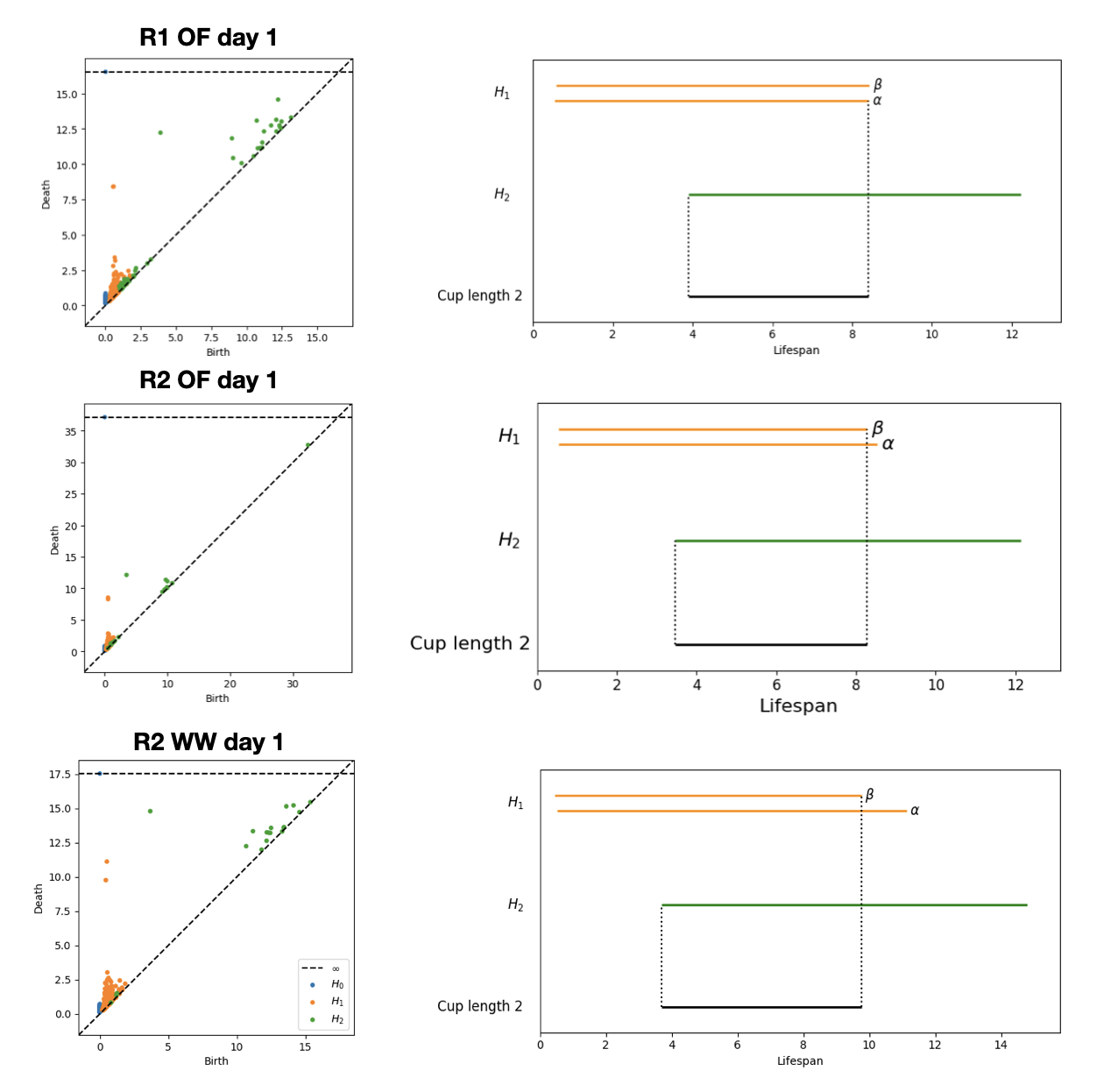}
  \caption{Persistence diagram (left) and cup-length computation (right) for R1 OF day 1, R2 OF day 1, and R2 WW day 1 modules. Persistence is computed using 500 landmarks. The cup-length computation shows two most persistent 1-dimensional bars (orange), the most persistent 2-dimensional bar (green), as well as the resulting cup-length-2 interval (black).}
  \label{fig:rats1}
\end{figure}

 \begin{figure}
  \centering
  \includegraphics[width=1\textwidth]{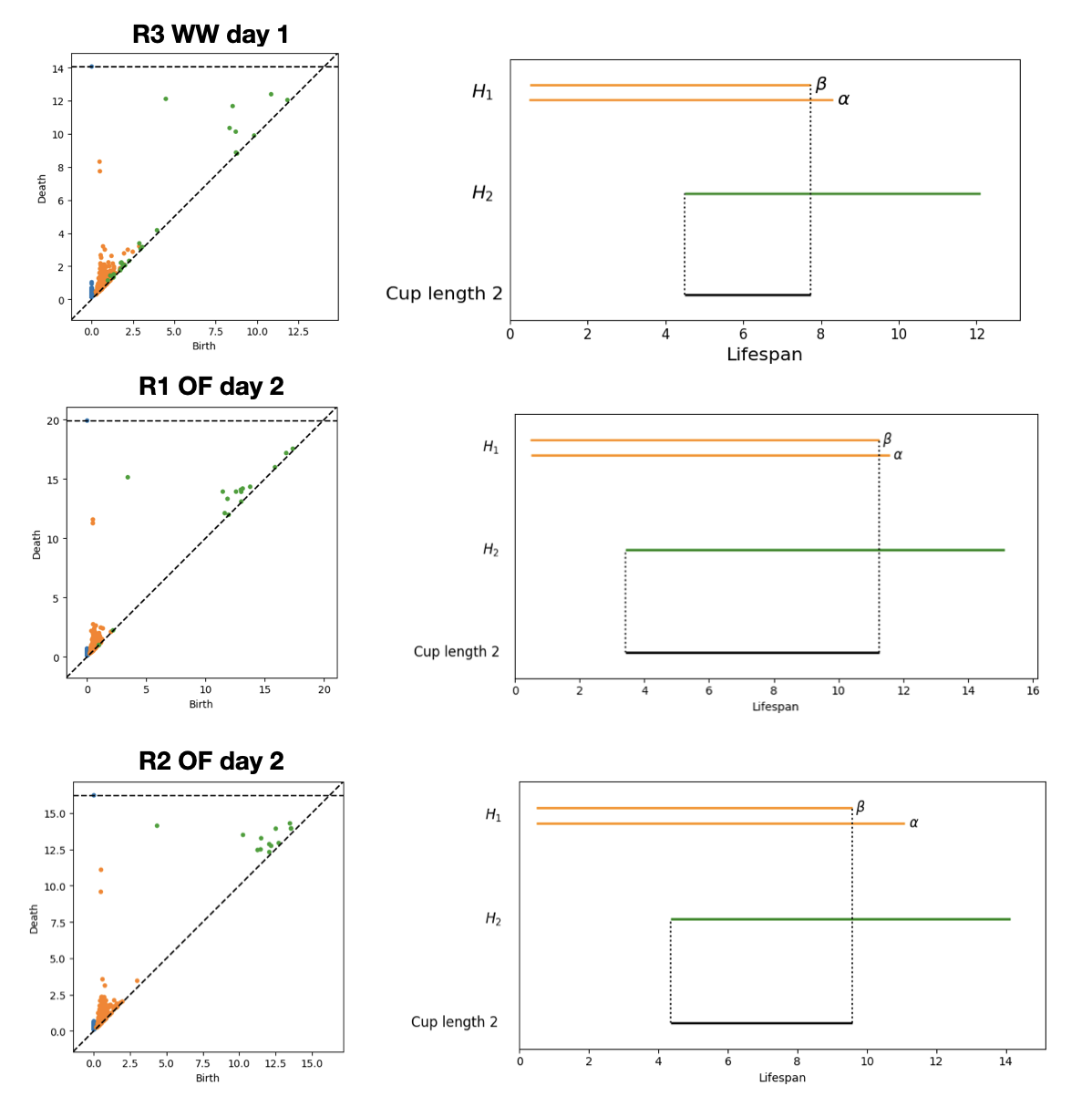}
  \caption{Persistence diagram (left) and cup-length computation (right) for R3 WW day 1, R1 OF day 2, and R2 OF day 2 modules. Persistence is computed using 500 landmarks. The cup-length computation shows two most persistent 1-dimensional bars (orange), the most persistent 2-dimensional bar (green), as well as the resulting cup-length-2-interval (black).}
  \label{fig:rats2}
\end{figure}

 \begin{figure}
  \centering
  \includegraphics[width=1\textwidth]{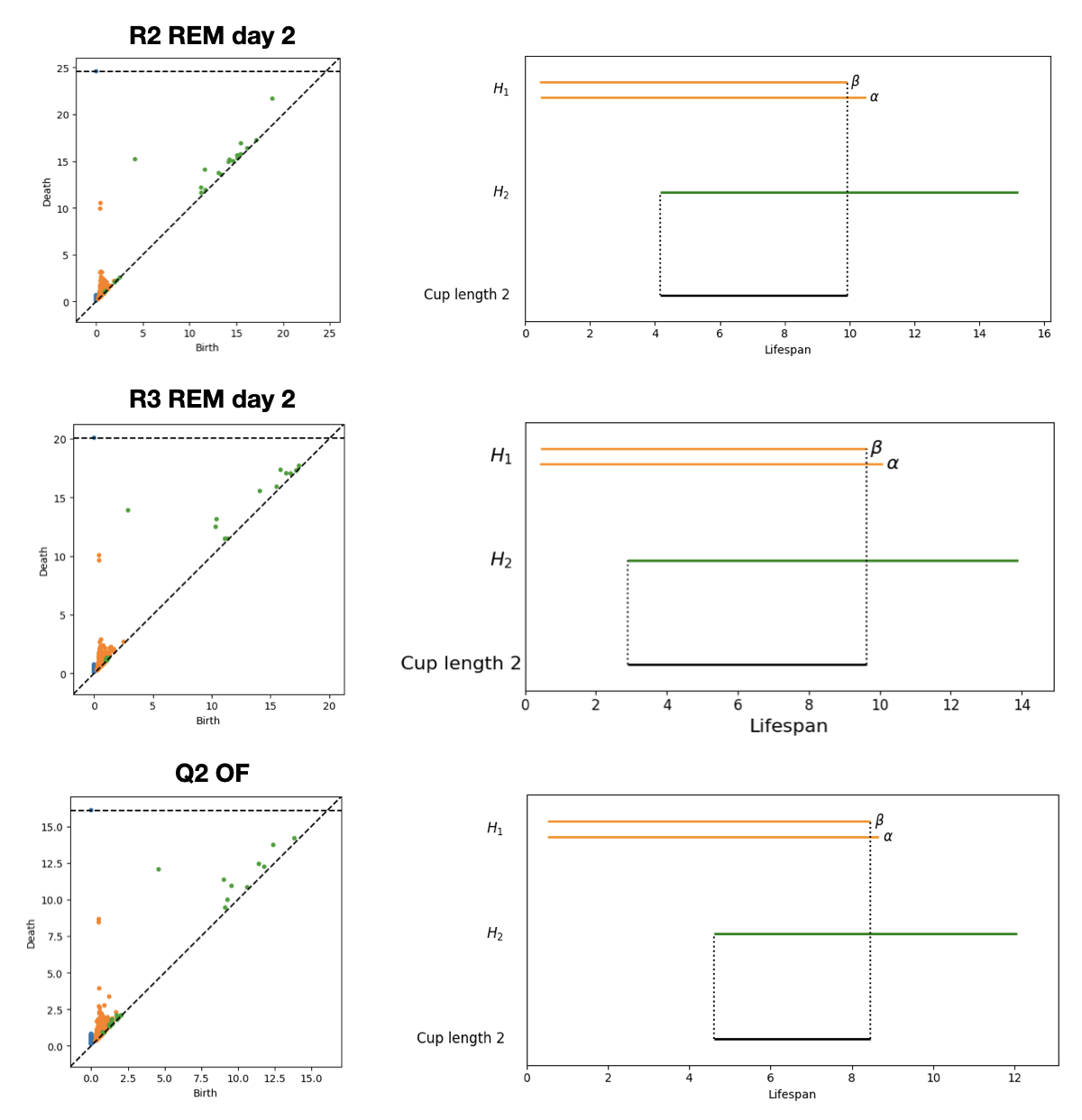}
  \caption{Persistence diagram (left) and cup-length computation (right) for R2 REM day 2, R3 REM day 2, and Q2 OF modules. Persistence is computed using 500 landmarks. The cup-length computation shows two most persistent 1-dimensional bars (orange), the most persistent 2-dimensional bar (green), as well as the resulting cup-length-2 interval (black).}
  \label{fig:rats3}
\end{figure}

 \begin{figure}
  \centering
  \includegraphics[width=1\textwidth]{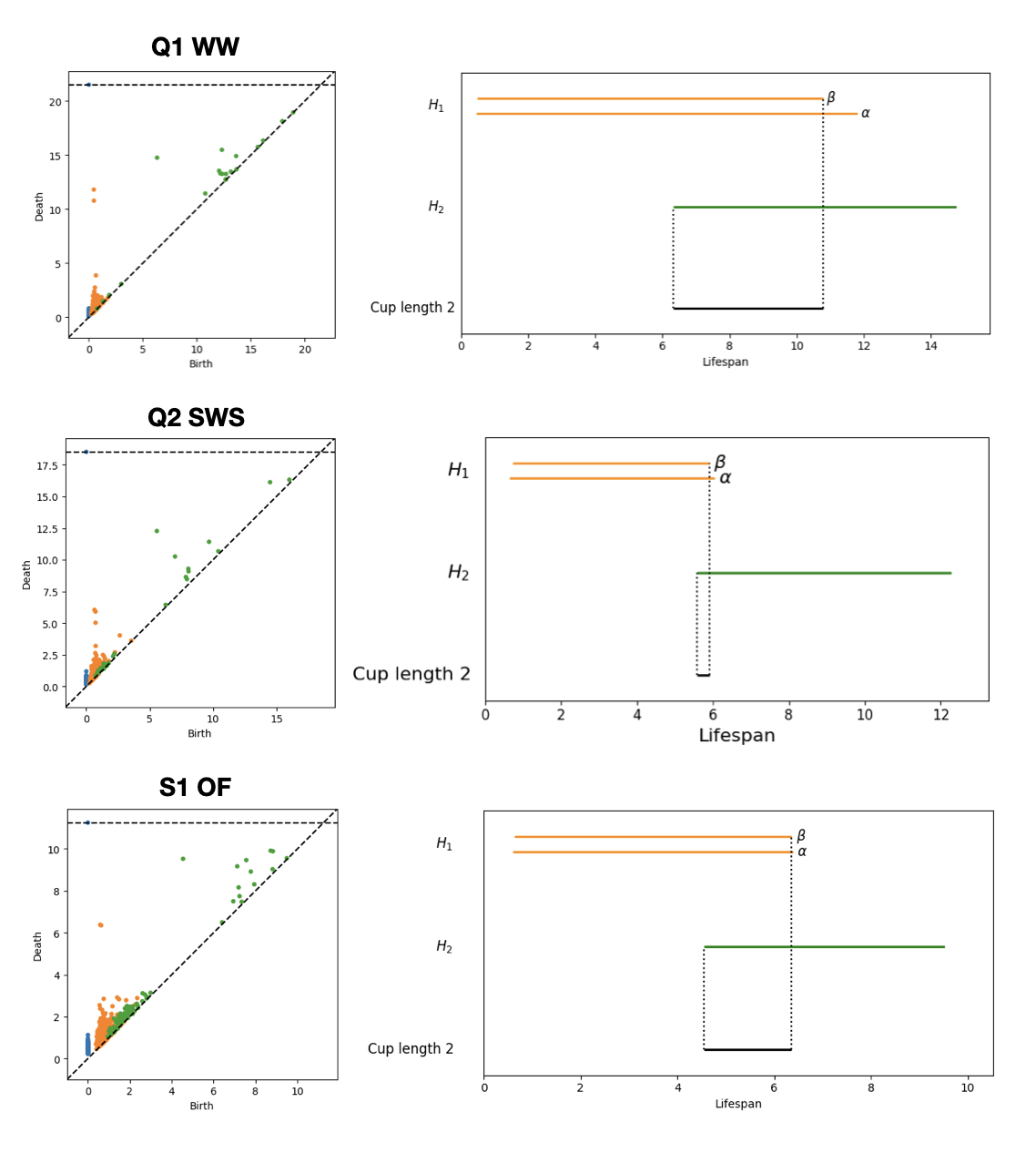}
  \caption{Persistence diagram (left) and cup-length computation (right) for Q1 WW, Q2 SWS, and S1 OF modules. Persistence is computed using 500 landmarks. The cup-length computation shows two most persistent 1-dimensional bars (orange), the most persistent 2-dimensional bar (green), as well as the resulting cup-length-2 interval (black).}
  \label{fig:rats4}
\end{figure}

 \begin{figure}
  \centering
  \includegraphics[width=1\textwidth]{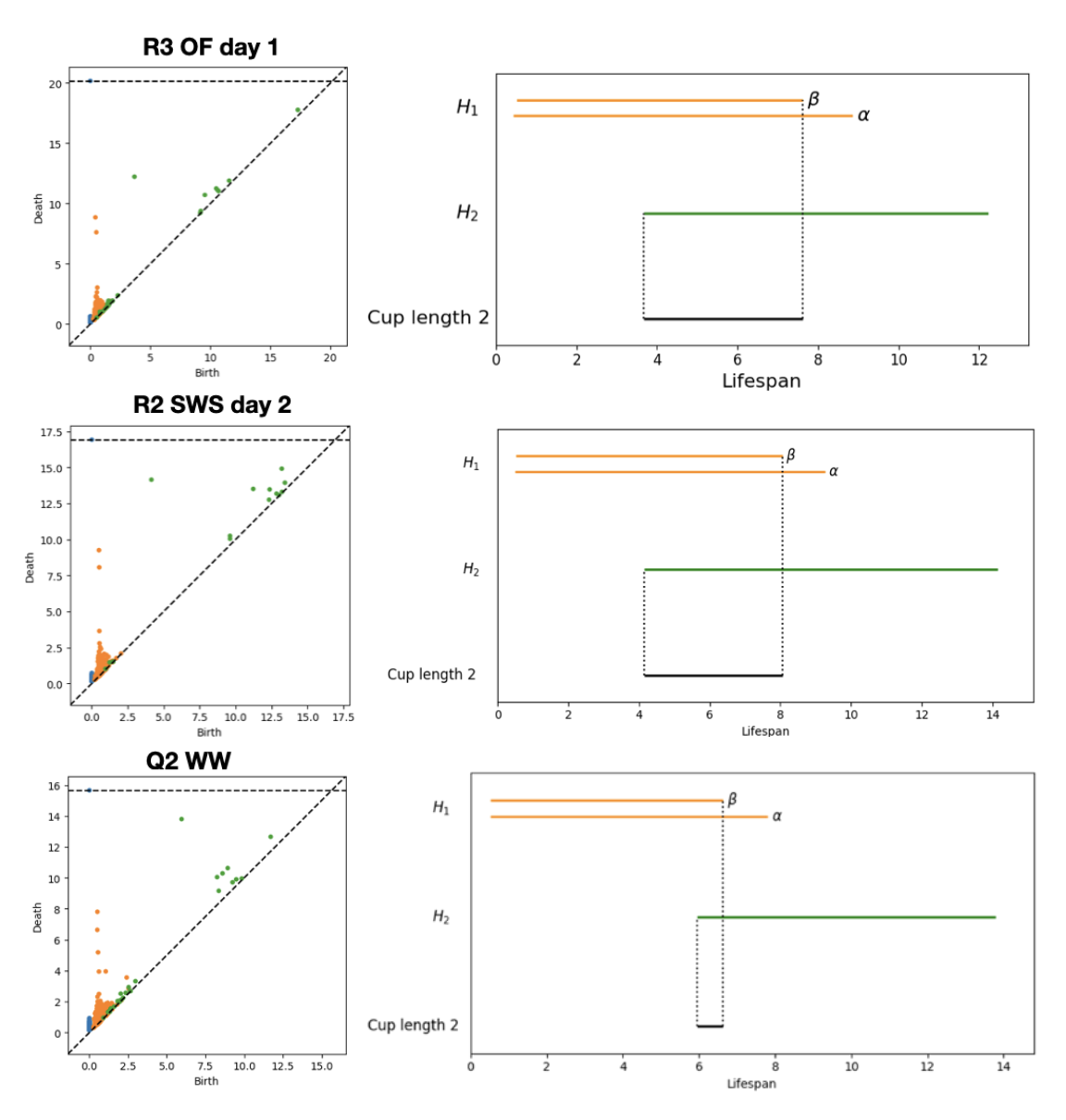}
  \caption{Persistence diagram (left) and cup-length computation (right) for R3 OF day 1, R2 SWS day 2, and Q2 WW modules. Persistence is computed using 500 landmarks. The cup-length computation shows two most persistent 1-dimensional bars (orange), the most persistent 2-dimensional bar (green), as well as the resulting cup-length-2 interval (black).}
  \label{fig:rats5}
\end{figure}

 \begin{figure}
  \centering
  \includegraphics[width=1\textwidth]{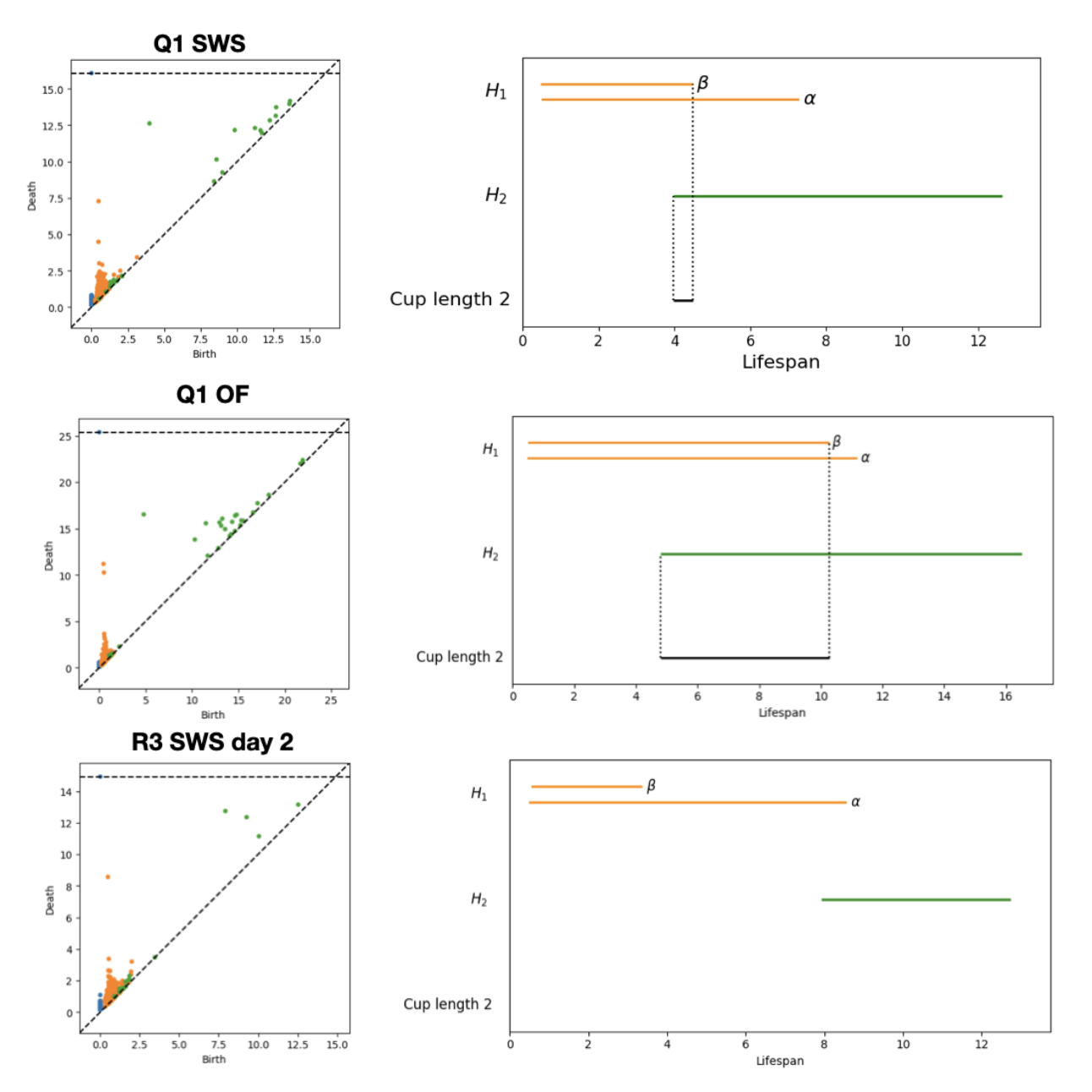}
  \caption{Persistence diagram (left) and cup-length computation (right) for Q1 SWS, Q1 OF, and R3 SWS day 2 modules. Persistence is computed using 500 landmarks. The cup-length computation shows two most persistent 1-dimensional bars (orange), the most persistent 2-dimensional bar (green), as well as the resulting cup-length-2 interval (black).}
  \label{fig:rats6}
\end{figure}

 \begin{figure}
  \centering
  \includegraphics[width=1\textwidth]{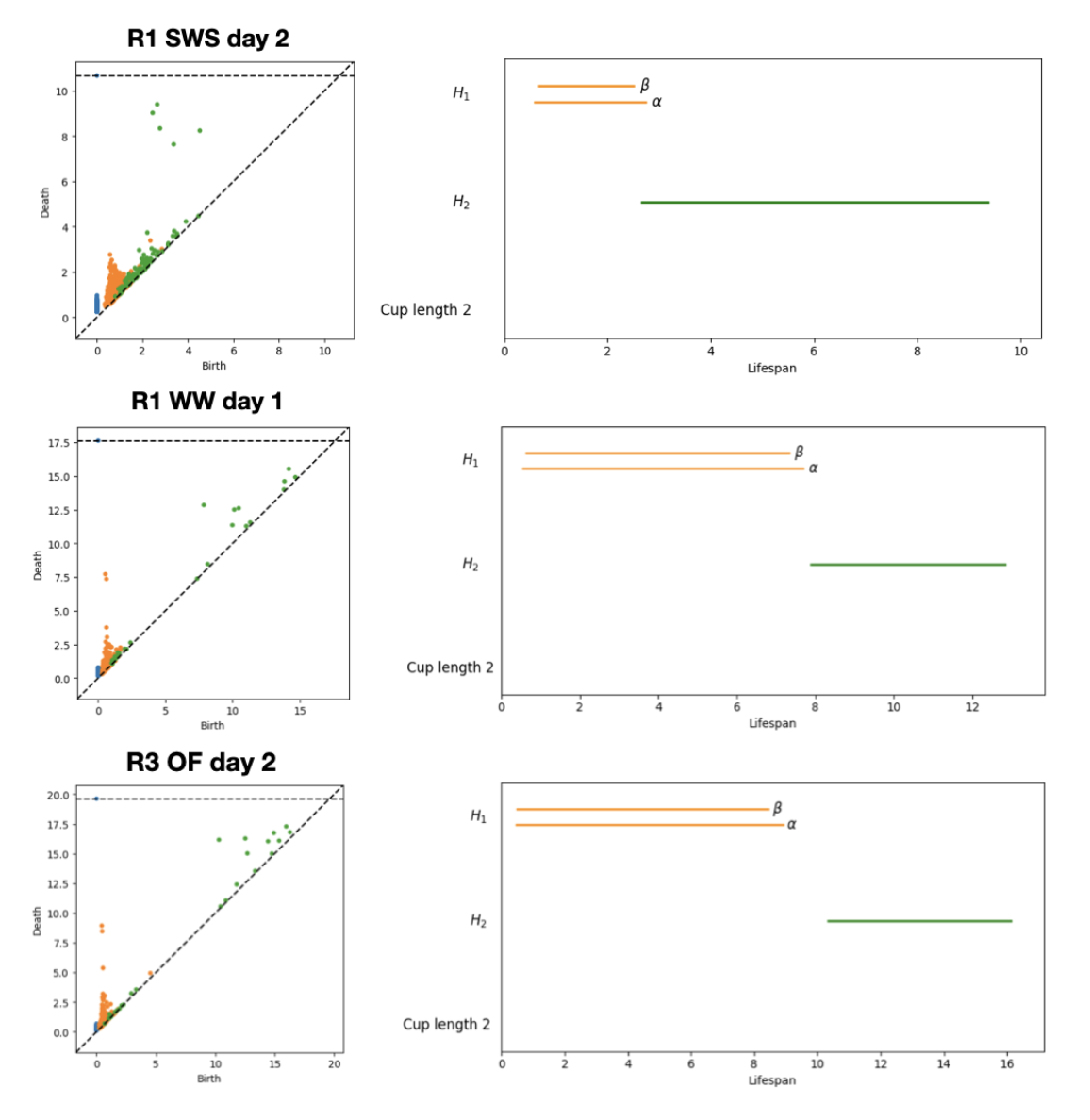}
  \caption{Persistence diagram (left) and cup-length computation (right) for R1 SWS day 2, R1 WW day 1, and R3 OF day 2 modules. Persistence is computed using 500 landmarks. The cup-length computation shows two most persistent 1-dimensional bars (orange), the most persistent 2-dimensional bar (green), as well as the resulting cup-length-2 interval (black).}
  \label{fig:rats7}
\end{figure}

 \begin{figure}
  \centering
  \includegraphics[width=1\textwidth]{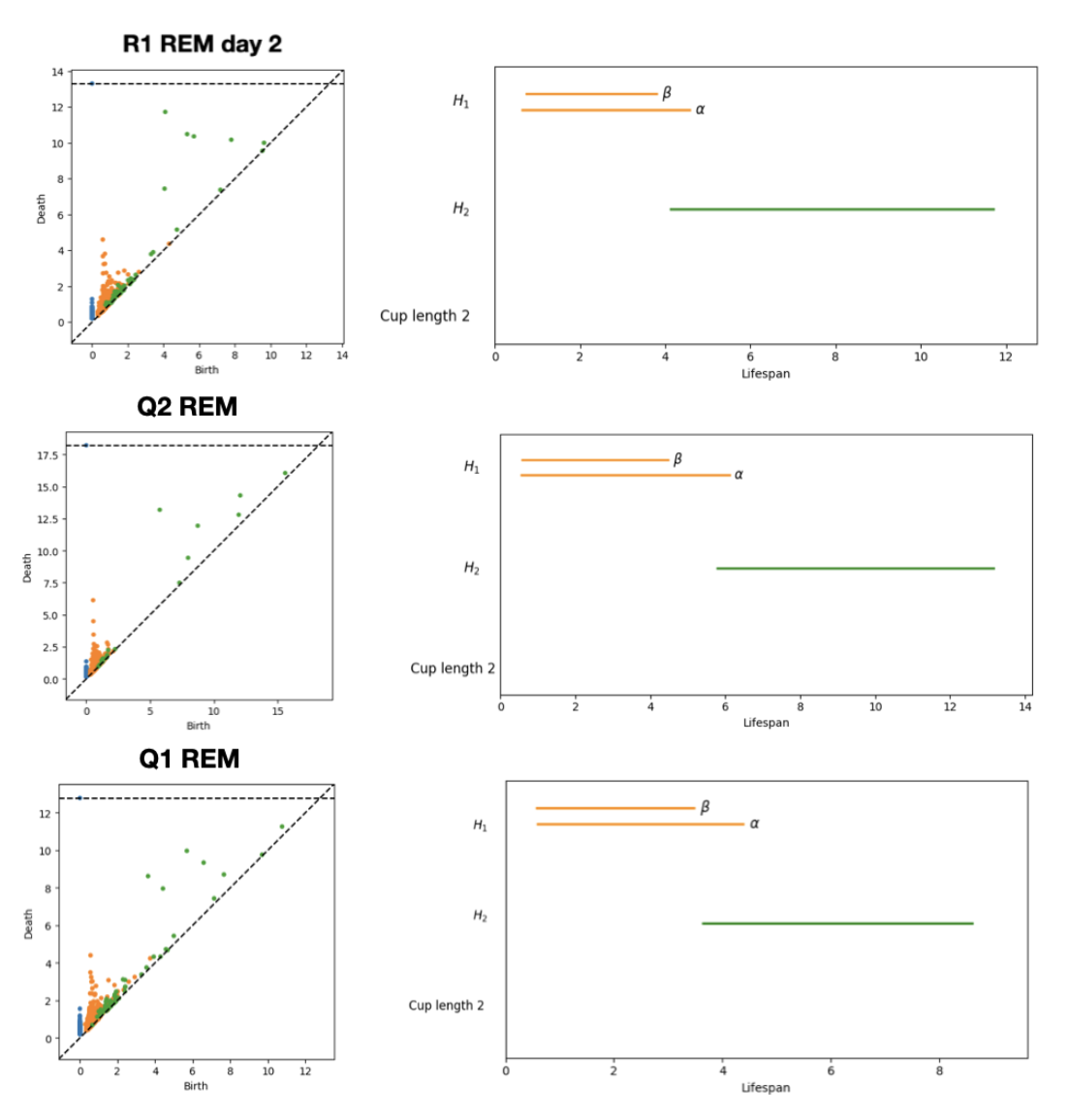}
  \caption{Persistence diagram (left) and cup-length computation (right) for R1 REM day 2, Q2 REM, and Q1 REM modules. Persistence is computed using 500 landmarks. The cup-length computation shows two most persistent 1-dimensional bars (orange), the most persistent 2-dimensional bar (green), as well as the resulting cup-length-2 interval (black).}
  \label{fig:rats8}
\end{figure}

 \begin{figure}
  \centering
  \includegraphics[width=1\textwidth]{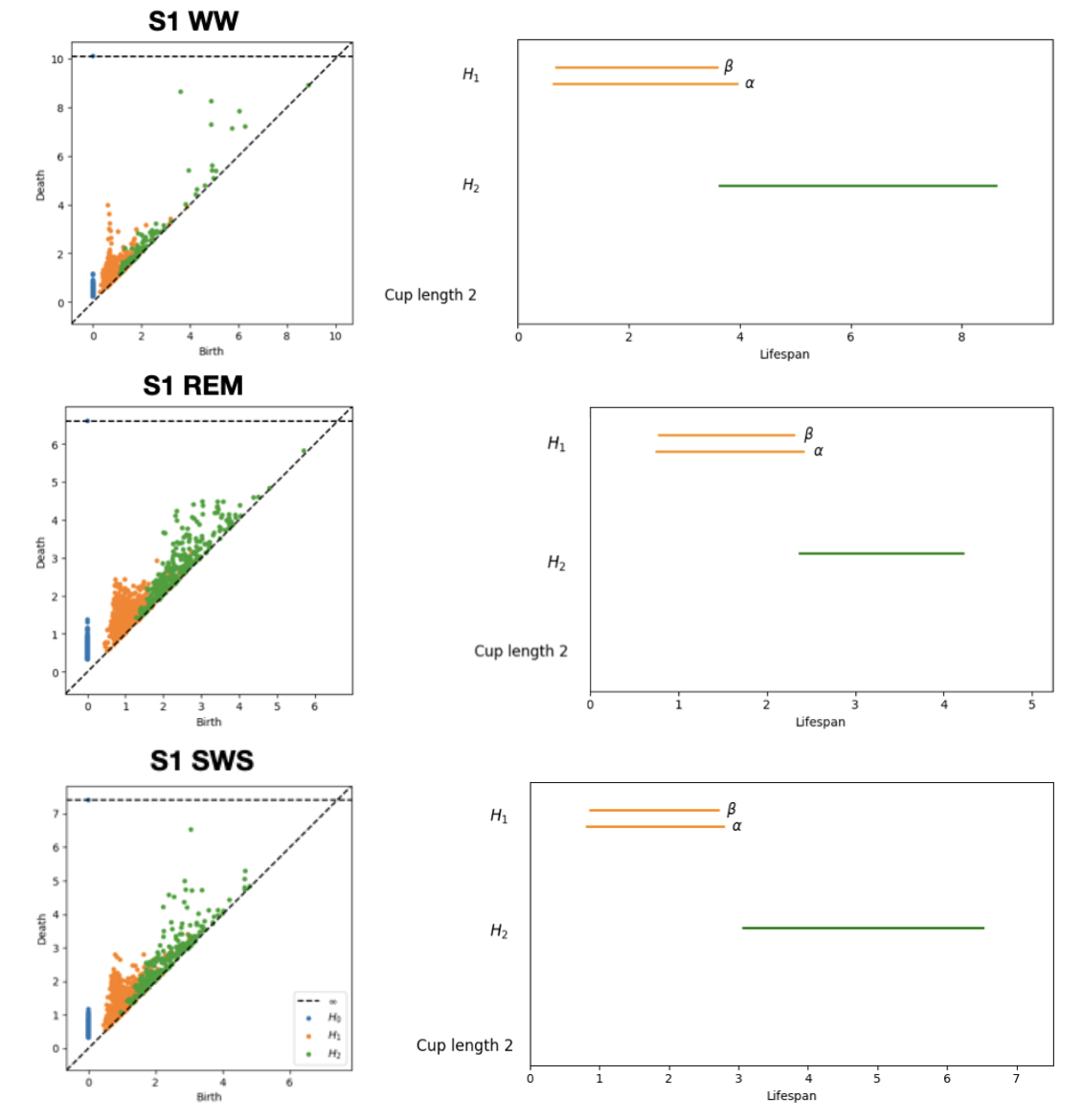}
  \caption{Persistence diagram (left) and cup-length computation (right) for S1 WW, S1 REM, and S1 SWS modules. Persistence is computed using 500 landmarks. The cup-length computation shows two most persistent 1-dimensional bars (orange), the most persistent 2-dimensional bar (green), as well as the resulting cup-length-2 interval (black).}
  \label{fig:rats9}
\end{figure}

\subsection{Code availability}
 
 The code is available on GitHub: \href{https://github.com/eivshina/persistent-cup-length}{https://github.com/eivshina/persistent-cup-length}.

\end{document}